\documentclass{amsart}

\usepackage{graphicx,epsfig}

\usepackage{amsmath,amssymb,latexsym, amsfonts, amscd, amsthm, xy}

\usepackage{url}

\usepackage{hyperref}

\usepackage{arevmath}

\usepackage{mathrsfs}

\usepackage[latin1]{inputenc}

\usepackage{tikz-cd}

\usepackage{enumitem, hyperref}
\makeatletter
\def\namedlabel#1#2{\begingroup
    #2%
    \def\@currentlabel{#2}%
    \phantomsection\label{#1}\endgroup
}
\makeatother

\input{xy}

\theoremstyle{plain}
\newtheorem{theorem}{Theorem}[section]

\newtheorem{proposition}[theorem]{Proposition}
\newtheorem{corollary}[theorem]{Corollary}

\newtheorem{lemma}[theorem]{Lemma}

\theoremstyle{definition}
\newtheorem{definition}[theorem]{Definition}
\newtheorem{remark}[theorem]{Remark}

\def\bA{\mathbb{A}}

\def\bC{\mathbb{C}}
\def\bF{\mathbb{F}}
\def\bH{\mathbb{H}}

\def\bN{\mathbb{N}}
\def\bK{\mathbb{K}}

\def\bP{\mathbb{P}}
\def\bQ{\mathbb{Q}}
\def\bZ{\mathbb{Z}}

\def\cC{\mathcal{C}}
\def\cD{\mathcal{D}}

\def\cF{\mathcal{F}}

\def\cI{\mathcal{I}}

\def\cJ{\mathcal{J}}

\def\cL{\mathcal{L}}
\def\cM{\mathcal{M}}

\def\cO{\mathcal{O}}

\def\cP{\mathcal{P}}

\def\cQ{\mathcal{Q}}

\def\cR{\mathcal{R}}

\def\cS{\mathcal{S}}

\def\cU{\mathcal{U}}
\def\cV{\mathcal{V}}
\def\cT{\mathcal{T}}

\def\fp{\mathfrak{p}}
\def\fq{\mathfrak{q}}

\def\inv{\mathrm{inv}}

\def\deg{\mathbf{deg}}

\def\Gal{\mathrm{Gal}}

\def\Norm{\mathrm{Norm}}

\def\Trace{\mathrm{Trace}}

\def\ulim{\mathbf{ulim}}

\def\alg{\mathbf{alg}}

\def\GF{\mathrm{GF}}

\def\red{\mathrm{red}}

\def\sign{\mathrm{sign}}

\def\sep{\mathbf{sep}}

\def\icard{\mathrm{icard}}

\begin{document}

\title[Defining the polynomial ring over an ultra-finite field]{Local symbols and a first-order definition of the polynomial ring over an ultra-finite field in its fraction field}

\author{Dong Quan Ngoc Nguyen}

\address{Department of Mathematics, George Washington University, DC 20052}

\date{October 16, 2023}

\email{\href{mailto:dongquan.ngoc.nguyen@gmail.com}{\tt dongquan.ngoc.nguyen@gmail.com}\\
\href{mailto:dongquan.nguyen@gwu.edu}{\tt dongquan.nguyen@gwu.edu}}

\maketitle

\tableofcontents

\begin{abstract}

In this paper, we prove the existence of a first-order definition of the polynomial ring over a nonprincipal ultraproduct of finite fields of unbounded cardinalities in its fraction field by a universal-existential formula in the language of rings augmented by an additional constant symbol $t$. As a consequence, we prove that the full first-order theory of the rational function field over a nonprincipal ultraproduct of finite fields of characteristic $0$ is undecidable.

\end{abstract}

\section{Introduction}
\label{sec-Introduction}

Hilbert's Tenth Problem is the tenth on the list of mathematical problems posed by David Hilbert \cite{Hilbert} in 1900 that asks for an algorithm which, for an arbitrary polynomial equation with integer coefficients and a finite number of unknowns, can decide whether the equation has an integral solution, that is, the positive existential theory $\mathrm{Th}_{\exists} \bZ$ of the integers $\bZ$ is undecidable. The work of Davis, Putnam, Robinson \cite{DPR1961}, and Matiyasevich \cite{matiyasevich} confirms that such an algorithm does not exist, and thus $\mathrm{Th}_{\exists}\bZ$ is in fact undecidable.

A natural generalization of Hilbert's Tenth Problem is to determine whether there exists an algorithm for the analogous problem with $\bZ$ replaced by other rings or fields such as the field $\bQ$ of rational numbers, number fields, or the rings of integers of number fields. The decidability of $\mathrm{Th}_{\exists}\bQ$ remains open. In \cite{rob49}, Robinson proves that the full first-order theory $\mathrm{Th}\bQ$ of $\bQ$ is undecidable by showing that there exists a first-order definition of $\bZ$ in $\bQ$, and thus the undecidability of $\mathrm{Th}\bQ$ follows from that of $\mathrm{Th}\bZ$. The strategy used in Robinson's Theorem 3.1 in \cite{rob49} suggests a \textit{definability} approach for Hilbert's Tenth Problem for $\bQ$: if there were an existential definition of $\bZ$ in $\bQ$, then the undecidability of $\mathrm{Th}_{\exists}\bQ$ easily follows from that of $\mathrm{Th}_{\exists}\bZ$ of $\bZ$.

The main aim of this paper is to prove  the existence of a first-order definition of the polynomial ring over a (nonprincipal) ultraproduct of finite fields of unbounded cardinalities \footnote{Our method also works when the ultraproduct is taken over the same finite field in which case the ultraproduct becomes a finite field itself. But a first-order definition of $\bF_q[t]$ in $\bF_q(t)$ is already known in \cite{EM18}, \cite{daans-21}, or \cite{tyrell}, and so for the sake of simplicity, we always work with an ultraproduct of finite fields of unbounded cardinalities.} in its fraction field.

\begin{theorem}
\label{thm-Main-Theorem-in-Nguyen-paper}
(See Corollary \ref{cor-main-corollary1})

Let $F$ be an ultraproduct of finite fields $\bF_{q_s}$ with $s \in S$, of unbounded cardinalities with respect to a nonprincipal ultrafilter $\cD$, where the $\bF_{q_s}$ are finite fields with $q_s$ elements and the $q_s$ are powers of primes $p_s$ for each $s \in S$. Then the polynomial ring $F[t]$ can be defined in the rational function field $F(t)$ by a first-order formula in $\cL_{\mathrm{rings}} \cup \{t\}$, where $\cL_{\mathrm{rings}}$ denotes the language of rings.

\end{theorem}

\begin{remark}

\begin{itemize}

\item [(i)] Suppose that the primes $p_s$, the characteristics of the finite fields $\bF_q$, are distinct for $\cD$-almost all $s \in S$. Then \L{}o\'s' theorem implies that $F$ has characteristic $0$. By the Lefschetz principle, the field $\bC$ of complex numbers is a nonprincipal ultraproduct of $\bF_{q_s}^{\alg}$, where the $\bF_{q_s}^{\alg}$ denotes the algebraic closure of the $\bF_{q_s}$. Thus $F(t)$ can be viewed as a subfield of $\bC(t)$ by embedding $F$ into $\bC$. The main motivation for studying definability questions for $F[t]$ in $F(t)$ stems from an analogue of Hilbert's Tenth Problem for $\bC(t)$ and the full first-order decidability problem of $\bC(t)$, both of which remains open.

\item [(ii)] Suppose that $p_s = p$ for some prime $p$ for $\cD$-almost all $s \in S$. Then $F$ is an infinite extension of the finite field $\bF_p$. In fact, by choosing an appropriate ultrafilter $\cD$, we can obtain an ultraproduct $F$ of characteristic $p > 0$ that contains the algebraic closure $\bF^{\alg}_p$ of $\bF_p$. Thus Theorem \ref{thm-Main-Theorem-in-Nguyen-paper} extends a first-order definability of $\bF_q[t]$ in $\bF_q[t]$ in \cite{EM18}, \cite{daans-21}, or \cite{tyrell} with $\bF_q$ being a finite extension of $\bF_p$ to an infinite extension of $\bF_p$ to obtain a first-order definability of $F[t]$ in $F(t)$, where $F$ is an infinite extension of $\bF_p$.

\end{itemize}

\end{remark}

Using R. Robinson's Theorem \cite{RROB}, we obtain the undecidability of the full first-order theory of $F(t)$, where $F$ is an ultraproduct of finite fields $\bF_{q_s}$ of characteristic $0$.

\begin{corollary}
\label{cor-main-cor-in-the-introduction}
(see Corollary \ref{cor-undecidability-of-GF(kappa)(t)})

Let $F$ be an ultraproduct of finite fields $\bF_{q_s}$ of characteristic $0$. Then the full first-order theory of $F(t)$ is undecidable.

\end{corollary}

Note that the ultraproduct $F$ in the above result can be embedded into $\bC$, and thus the above result provides an infinite family of subfields $F$ of $\bC$ such that the full first-order theory of $F(t)$ is undecidable.

In order to prove Theorem \ref{thm-Main-Theorem-in-Nguyen-paper}, we follow a similar strategy of the proof of Tyrrell \cite[Corollary 2.9]{tyrell} (see also \cite{tyrell-thesis}) in which he provides a new universal definition of $\bF_q[t]$ in $\bF_q[t]$, where $\bF_q$ denotes the finite field with $q$ elements. The methods used in Tyrrell \cite{tyrell} follow the ideas and strategy in the work of Poonen \cite{poonen-2009} and Koenigsmann \cite{koe-2016}. Before the work of Tyrrell \cite{tyrell}, the universal definability of $\bF_q[t]$ in $\bF_q(t)$ had been proved by Eisentr\"ager and Morrison \cite{EM18} in which they adapted Park's universal definition of the ring $\cO_K$ of integers of a number field $K$ (see Park \cite{park}) to the function field context. Daans \cite{daans-21} unifies the work of Park \cite{park}, and Eisentr\"ager and Morrison \cite{EM18} to prove the universal definability of the ring $\cO_{S, K}$ of $S$-integers in a global field $K$, where $S$ is a nonempty finite set of primes of $K$. The strategy of proving definability results in the work of Park \cite{park}, Eisentr\"ager and Morrison \cite{EM18}, Daans \cite{daans-21}, and Tyrrell \cite{tyrell} follows the work of Poonen \cite{poonen-2009} that is refined in a breakthrough paper by Koenigsmann \cite{koe-2016} in which he establishes a first universal definition of $\bZ$ in $\bQ$. A key ingredient in the proofs of these papers is the formula given by Serre \cite[Proposition 8, p.210]{serre-local-fields} for local symbols over a field $K$ that is complete under a discrete valuation such that its residue field is quasi-finite (see Subsection \ref{subsec-local-symbols} for its notion). In the case of a global field $K$ that is considered in \cite{koe-2016, park, EM18, daans-21, tyrell}, the residue fields of $\fp$-adic completions of $K$ are all finite fields which are quasi-finite, and there is an \textit{explicit} formula for local symbols given by Serre \cite[Corrollary to Lemma XIV 3.2]{serre-local-fields} that only depends on the $\fp$-adic valuation and the finite cardinality of the residue field. Such an explicit formula for local symbols provides an effective tool for studying ramification of certain quaternion algebras that are used to provide a universal definition of the ring of integers of the global field $K$.

In the case of the rational function field $\bK = F(t)$ studied in this paper, with $F$ as in Theorem \ref{thm-Main-Theorem-in-Nguyen-paper}, the residue fields of $P$-adic completions $\bK_P$ of $\bK$ are algebraic finite extensions of $F$ that are also quasi-finite (see Ax \cite{ax-1968}). But in contrast with the case of global fields, the residue fields in our case are \textit{uncountable}, and thus a similar explicit formula for local symbols over $\bK_P$ was not known. In \cite[Proposition 8, p.210]{serre-local-fields}, Serre establishes a formula for such local symbols in terms of a topological generator of the absolute Galois group of the quasi-finite residue field which is isomorphic to the profinite completion $\widehat{\bZ}$ of $\bZ$. But Serre's formula in \cite[Proposition 8, p.210]{serre-local-fields} is not \textit{explicit enough} to compute the values of local symbols which in turn verifies ramification of the quaternion algebras associated to them. In Nguyen \cite{nguyen-ultrafinite-fields-2023}, the present author establishes an explicit higher reciprocity law for $\bK = F(t)$ which use the \textit{internal cardinalities} (see Subsection \ref{subsec-internal-sets-and-cardinality} for its notion) that are hypernatural numbers, as a way of \textit{counting} the number of elements of the residue fields of $P$-adic completions of $\bK$. One of the main contributions of this paper is to use the explicit higher reciprocity law in Nguyen \cite{nguyen-ultrafinite-fields-2023} to establish an explicit formula for local symbols over $P$-adic completions $\bK_P$ of $\bK$ which, in analogy with the formula given by Serre \cite[Corrollary to Lemma XIV 3.2]{serre-local-fields} for global fields, also depends on the $P$-adic valuations, but the internal cardinality replaces the cardinality in the formula. Combining the explicit formula for local symbols given in this paper with the higher reciprocity law in \cite{nguyen-ultrafinite-fields-2023}, we prove an analogue of the Hilbert reciprocity law for the rational function field $\bK = F(t)$. We then use both the explicit formula for local symbols and the analogue of the Hilbert reciprocity law for $\bK$ to study ramification of certain quaternion algebras over $\bK$ that are used to provide a universal-existential formula for the polynomial ring over an ultraproduct of finite fields of unbounded cardinalities in its fraction field.

The structure of this paper is as follows. In Section \ref{sec-basic-notions}, we recall some basic notions in model theory such as ultraproducts, hypernatural numbers and hyperintegers, internal cardinality, and the action of hyperintegers on ultraproducts of fields that will be used throughout the paper. In Section \ref{sec-local-symbols}, we recall the notion of local symbols given in Serre \cite{serre-local-fields}, and prove an explicit formula for local symbols over $P$-adic completions of the rational function field $\bK = F(t)$ over an ultraproduct of finite fields of distinct characteristics (see Theorem \ref{thm-Nguyen-formula-for-local-symbols-K_P}). We also prove an analogue of the Hilbert reciprocity law for $\bK$ in this section (see Theorem \ref{thm-analogue-of-the-Hilbert-reciprocity-law}). Finally in Section \ref{sec-definition-of-A-in-K}, we prove Theorem \ref{thm-Main-Theorem-in-Nguyen-paper} (see Corollary \ref{cor-main-corollary1}) that provides a first-order definition of the polynomial ring $\bA = F[t]$ over an ultraproduct of finite fields of unbounded cardinalities in its fraction field $\bK = F(t)$ by a universal-existential formula. Note that in Corollary \ref{cor-main-corollary1}, we also provide a number of universal and existential quantifiers needed for the definition of $\bA = F[t]$ in $\bK = F(t)$, but do not attempt to reduce the number of quantifiers to the smallest possible one.

\section{Basic notions and notation}
\label{sec-basic-notions}

In this section, we recall some basic notions and fix some notation that we will use throughout the paper. The main references for this section are Bell--Slomson \cite{bell-slomson}, Rothmaler \cite{rothmaler}, and Schoutens \cite{schoutens}.

For any subring $A$ of a commutative ring $R$, we will always denote $A^{\times}$ the set of units in $A$, i.e., the set of elements $a \in A$ such that there exists an element $b \in A$ for which $ab = 1$.

For a given field $k$, we always denote by $k^{\sep}$ a separable closure of $k$. We use the notation $k^{\alg}$ for an algebraic closure of $k$.

\subsection{Ultraproducts and \L{}o\'s' theorem}

Throughout this paper, we fix an infinite set $S$. A collection $\cD$ of infinite subsets of $S$ is called a \textit{nonprincipal ultrafilter on $S$} if $\cD$ is closed under finite intersection and for any subset $A \subseteq S$, either $A$ or its complement $S \setminus A$ belong in $\cD$. Since every set in $\cD$ must be infinite, we deduce that any cofinite subset of $S$ belongs in $\cD$. Throughout this paper, we always fix $\cD$ to be a nonprincipal ultrafilter on $S$.

For each $s \in S$, let $A_s$ be a set. Let $\widehat{A} = \prod_{s \in S}A_s$ be the Cartesian product of the $A_s$. We define an equivalence relation on $A$ as follows: for elements $\widehat{a} = (a_s)_{s\in S}$, $\widehat{b} = (b_s)_{s \in S}$ in $\widehat{A}$, $\hat{a}$ is equivalent to $\widehat{b}$ if the set $\{s \in S \; : \; a_s = b_s\}$ belongs in $\cD$. Following Schoutens \cite{schoutens}, we will also denote the equivalence class of an element $(a_s)_{s \in S} \in \widehat{A}$ by $\ulim_{s \in S}a_s$.

The set of all equivalence classes on $\widehat{A}$ is called the \textit{ultraproduct of the $A_s$} with respect to the ultrafilter $\cD$, and is denoted by $\prod_{s \in S}A_s/\cD$. In this paper, only ultraproducts with respect to a nonprincipal ultrafilter will be considered, and so whenever we consider \textit{ultraproducts}, we mean \textit{nonprincipal ultraproducts}.

If all the sets $A_s$ are equal to the same set $A$, we call the ultraproduct of the $A_s$ the \textit{ultrapower of $A$}, and use $A^{\#}$ to denote the ultrapower. There is a canonical map from $A$ to $A^{\#}$ that sends each element $a \in A$ to the equivalence class $\ulim_{s \in S}a \in A^{\#}$ in which all components are equal to $a$. We identify $A$ with its image in $A^{\#}$ under the canonical map, and consider $A$ as a subset of $A^{\#}$.

Throughout this paper, whenever we say a certain property (P) holds for $\cD$-almost all $s \in S$, we mean the set $\{s \in S\; : \; \text{(P) holds in $A_s$}\}$ belongs in the ultrafilter $\cD$.

If all the sets $A_s$ have the same algebraic structure such as groups, rings, or fields, their ultraproduct $\bA = \prod_{s \in S}A_s/\cD$ inherits the same algebraic structure from its components as follows. Let $\star_s : A_s \times A_s \to A_s$ be a binary operation for $\cD$-almost all $s \in S$. We always write $a_s \star_s b_s$ instead of $\star_s(a_s, b_s)$. The map $\star : \bA \times \bA \to \bA$ defined by
\begin{align}
\label{e1-operation-ultraproduct}
(\ulim_{s \in S}a_s)\star (\ulim_{s \in S}b_s) = \ulim_{s \in S}(a_s\star_s b_s) \in \bA
\end{align}
is a binary operation on $\bA$. Note that the operation $\star$ does not depend on the choices of $a_s$ and $b_s$ for the equivalence classes $\ulim_{s \in S}a_s$, $\ulim_{s \in S}b_s$. One can verify that if $A_s$ are groups with multiplication $\star_s$ and identity element $1_s$, then $\bA$ is a group with multiplication $\star$ and identity element $1 = \ulim_{s \in S}1_s$. Furthermore if $A_s$ is abelian for $\cD$-almost all $s \in S$, then $\bA$ is also abelian.

We obtain the following.
\begin{proposition}
\label{prop-notion-section-algebraic-structures-of-ultraproducts}

We maintain the same notation as above.

\begin{itemize}

\item [(i)] If $A_s$ is a (commutative) ring (resp. field) with respect to addition $+_s$ and multiplication $\cdot_s$ for $\cD$-almost all $s \in S$, then $\bA$ is a (commutative) ring (resp. field) with addition $+$ and multiplication $\cdot$ that are defined by (\ref{e1-operation-ultraproduct}).

\item [(ii)] If $A_s$ is a ring and $I_s$ is an ideal in $A_s$ for $\cD$-almost all $s \in S$, then $\prod_{s \in S}I_s/\cD$ is an ideal in $\bA$. Furthermore $\prod_{s \in S}(A_s/I_s)/\cD$ is isomorphic to $\bA/\left(\prod_{s \in S}I_s/\cD\right)$. The ideal $\prod_{s\in S}I_s/\cD$ is generated by $h$ elements if $A_s$ is generated by $h$ elements for $\cD$-almost all $s\in S$.

\end{itemize}

\end{proposition}

\begin{proof}

See Schoutens \cite[2.1.4, p.10]{schoutens} for the proof.

\end{proof}

We state one of the most fundamental transfer principle for ultraproducts.

\begin{theorem}
(\L{}o\'s' theorem)
\label{thm-Los}

Let $\cR$ be a ring, let $A_s$ be an $\cR$-algebra for each $s \in S$, and let $\bA$ be their ultraproduct. Let $\psi(\zeta_1, \ldots, \zeta_m)$ be a formula with parameters from $\cR$ whose free variables are among $\zeta_1, \ldots, \zeta_m$. For each $s \in S$, let $\widehat{a}_s$ be an $m$-th tuple in $A_s$, and let $\widehat{a}$ be their ultraproduct which is an $m$-th tuple in $\bA$. Then $\psi(\widehat{a}_s)$ holds in $A_s$ for $\cD$-almost all $s \in S$ if and only if $\psi(\widehat{a})$ holds in $\bA$.

\end{theorem}

\begin{proof}

For the proof of \L{}o\'s' theorem, see \cite{bell-slomson} or \cite{rothmaler}.

\end{proof}

\subsection{Hyperintegers and hypernaturals}

The ring of integers is denoted by $\bZ$. The set of natural numbers, i.e. positive integers $1, 2, \ldots$ is denoted by $\bN$.

Throughout this paper, we let $\bZ^{\#}$ be the ultrapower $\prod_{s \in S}\bZ/\cD$ whose elements are called \textbf{hyperintegers}. We also let $\bN^{\#}$ be the ultrapower $\prod_{s \in S}\bN/\cD$ whose elements are called \textbf{hypernatural numbers} or simply \textbf{hypernaturals}. Since $\bN \subset \bZ$, $\bN^{\#} \subset \bZ^{\#}$.

\subsubsection{} There is a total order on $\bZ^{\#}$. For any hyperintegers $a = \ulim_{s \in S}a_s$, $b = \ulim_{s \in S}b_s$, we define $a \le b$ if and only if $a_s \le b_s$ for $\cD$-almost all $s \in S$, i.e., $\{s \in S \; : \; a_s \le b_s\} \in \cD$. We further write $a < b$ if $a \le b$ and $a \ne b$.

A hyperinteger $a = \ulim_{s \in S}a_s$ with integers $a_s$ is \textbf{positive} if $a > 0$, i.e., $a_s > 0$ for $\cD$-almost all $s \in S$. A positive hyperinteger is also a hypernatural number. A hyperinteger $a$ is \textbf{negative} if $a < 0$. We further say that $a$ is \textbf{nonnegative} if $a \ge 0$.

The set of positive (resp., negative and nonnegative) hyperintegers is denoted by $\bZ^{\#}_{>0}$ (resp. $\bZ^{\#}_{< 0}$, and $\bZ^{\#}_{\ge 0}$).

\subsection{Internal sets and cardinality}
\label{subsec-internal-sets-and-cardinality}

In this subsection, we recall the notion of internal cardinality (see, for example, Goldblatt \cite{goldblatt}) that plays a key role in this paper.

\begin{definition}
\label{def-internal-cardinality}

Let $\alpha = \ulim_{s \in S}\alpha_s$ be a hypernatural in $\bN^{\#}$, where the $\alpha_s$ are natural numbers.

A set $A$ is said to have \textbf{internal cardinality $\alpha$} if there exists a collection of finite sets $A_s$ of cardinality $\alpha_s$ for $\cD$-almost all $s \in S$ such that $A = \prod_{s\in S} A_s/\cD$, i.e., $A$ is the ultraproduct of the sets $A_s$ with respect to $\cD$, and for $\cD$-almost all $s \in S$, $A_s$ is a finite set of $\alpha_s$ elements. In this case, we also call $A$ an \textbf{internal set}.

In symbol, we write $\icard(A)$ for the internal cardinality of $A$.

\end{definition}

\subsection{The action of hyperintegers on ultraproducts}
\label{subsec-action-of-hyperintegers-on-ultraproducts}

Let $A = \prod_{s \in S}A_s/\cD$ be a (nonprincipal) ultraproduct of rings $A_s$, and let $n = \ulim_{s \in S}n_s$ be a nonnegative integer in $\bZ^{\#}_{\ge0}$ for some nonnegative integers $n_s$. The \textbf{ultrapower map on $A$ with exponent $n$} is defined as follows. For an element $a = \ulim_{s \in S}a_s \in A$ for some $a_s \in A_s$, define
\begin{align}
\label{e1-ultra-power-with-exponent-in-hypernaturals}
a^n := \ulim_{s \in S}a_s^{n_s} \in A.
\end{align}
It is easy to check that the above definition does not depend on the choice of $a_s$ for the equivalence class $a \in A$. By \L{}o\'s' Theorem \ref{thm-Los}, we can show that the ultrapower map satisfies the same properties as the standard power map as follows: for any elements $a = \ulim_{s \in S}a_s, b = \ulim_{s \in S}b_s \in A$ and any nonnegative integers $m = \ulim_{s \in S}m_s, n = \ulim_{s \in S}n_s$,
\begin{itemize}

\item [(i)] $(ab)^n = a^n \cdot b^n$;

\item [(ii)] $a^n\cdot a^m = a^{m + n}$; and

\item [(iii)] $(a^n)^m = a^{mn}$.

\end{itemize}

Similarly we can define the \textbf{ultrapower map on $A^{\times} = \prod_{s \in S}A_s^{\times}/\cD$ with exponent as a hyperinteger $\alpha = \ulim_{s\in S}n_s \in \bZ^{\#}$} in the same way as (\ref{e1-ultra-power-with-exponent-in-hypernaturals}) above. That is, for a hyperinteger $\alpha = \ulim_{s \in S}n_s \in \bZ^{\#}$ for some integers $n_s \in \bZ$ and an arbitrary element $a = \ulim_{s \in S}a_s \in A^{\times}$ for some elements $a_s \in A_s^{\times}$, define
\begin{align}
\label{e2-ultrapower-with-hyperinteger-exponent}
a^{\alpha} := \ulim_{s \in S}a_s^{n_s}.
\end{align}
Since $a \in A^{\times}$, $a_s \ne 0$ for $\cD$-almost all $s \in S$, and thus the above definition is well-defined. Thus we obtain the action of $\bZ^{\#}$ on the multiplicative group $A^{\times}$ of the form
\begin{align*}
\bZ^{\#} \times A^{\times} \to A^{\times}
\end{align*}
that sends each pair $(\alpha, a)$ to $a^{\alpha}$ defined as in (\ref{e2-ultrapower-with-hyperinteger-exponent}).

\section{Cup Products and Local Symbols}
\label{sec-local-symbols}

In this section, we recall the notion of local symbols over a field $K$ that is complete under a discrete valuation such that its residue field is quasi-finite. We then establish an explicit formula for local symbols for $P$-adic completions of the rational function field $\bK$ over an ultra-finite field (see Subsection \ref{subsec-local-symbols-when-k-is-an-ultra-finite-field} for its notion), and prove an analogue of the Hilbert reciprocity law for $\bK$. We also recall the higher reciprocity law in Nguyen \cite{nguyen-ultrafinite-fields-2023} that we will need in the proof of our main theorem.

\subsection{Local Symbols $(\chi, \beta)$ }
\label{subsec-local-symbols-chi-beta}

Let $K$ be a field, and let $K^{\sep}$ be a separable closure of $K$. Let $G_K= \Gal(K^{\sep}/K)$ denote the Galois group of $K^{\sep}$ over $K$.

Throughout this section, we denote by $H^n(G_K, \Gamma)$ the $n$-th cohomology group of $G_K$ with coefficients in $\Gamma$,  where $\Gamma$ is a $G_K$-module.

The exact sequence of $G_K$-modules
\begin{align*}
0 \rightarrow \bZ \rightarrow \bQ \rightarrow \bQ/\bZ \rightarrow 0,
\end{align*}
induces the exact sequence
\begin{align*}
1 \rightarrow H^0(G_K, \bZ) \rightarrow H^0(G_K, \bQ) \rightarrow H^0(G_K, \bQ / \bZ)\rightarrow  \\
H^1(G_K, \bZ) \rightarrow H^1(G_K, \bQ) \rightarrow H^1(G_K, \bQ / \bZ) \xrightarrow{\delta} H^2(G_K, \bZ),
\end{align*}
where $\delta$ is the $1$st connecting map.

Recall that an element $\chi$ in $H^1(G_K, \bQ / \bZ)$ is called a character of $G_K$. For such a character $\chi \in H^1(G_K, \bQ / \bZ)$, $\delta(\chi)$ is an element of $H^2(G_K, \bZ)$.

There is a cup product
\begin{align*}
\cup : H^0(G_K, {K^{\sep}}^\times) \times H^2(G_K, \bZ) \to H^2(G_K, {K^{\sep}}^\times),
\end{align*}
(see \cite{serre-local-fields}). Thus, for each $\beta \in K^\times = H^0(G_K, {K^{\sep}}^\times)$, the cup product $[\beta] \cup [\delta(\chi)]$ is an element of the Brauer group $\text{Br}(K)= H^2(G_K, {K^{\sep}}^\times)$.

We denote the element $[\beta] \cup [\delta(\chi)]$ by the symbol $(\chi, \beta)$. So the symbol $(\cdot, \cdot)$ is a map from $H^1(G_K, \bQ / \bZ) \times K^\times$ to $\text{Br}(K)= H^2(G_K, {K^{\sep}}^\times)$.

\subsection{The symbol $(\alpha, \beta)$}
\label{subsec-symbol-alpha-beta}

Let $n$ be a positive integer such that $K$ is of characteristic $0$ or $n$ is a prime to the characteristic of $K$. We further assume that $K$ contains the group $\mu_n$ of $n$-th roots of unity.

Take an arbitrary $\alpha \in K^\times$, and choose a root $\lambda \in K^{\sep}$ of the equation $\lambda^n = \alpha$, i.e., $\lambda$ is an $n$th root of $\alpha$ in $K^{\sep}$.

Fix a primitive $n$-th root of unity $\zeta_n \in \mu_n$. By identifying $\zeta_n$ with $1+n\bZ \in \bZ/ n\bZ$, one can identify $\mu_n$ with $ \bZ/ n\bZ$, and
\begin{align*}
\psi_\alpha(\sigma) = \frac{\sigma(\lambda)}{\lambda} \in \mu_n \cong \bZ/ n\bZ, \; \sigma \in G_K,
\end{align*}
 defines a homomorphism from $G_K$ to $\bZ/ n\bZ$.

For each $\sigma \in G_K$, we then set
\begin{align*}
\chi_\alpha(\sigma) =\frac{1}{n}\psi_\alpha(\sigma) \in  \left( \frac{1}{n} \right) \bZ/ \bZ,
\end{align*}
and then $\chi_\alpha$ is a homomorphism of $G_K$ into $\bQ / \bZ$, i.e., $\chi_\alpha \in H^1(G_K, \bQ/\bZ)$--a character of $G_K$. One can  verify that the map $\alpha \mapsto \chi_\alpha$ defines an isomorphism of $K^\times / K^{\times^n}$ onto the group of those characters of $G_K$ having order dividing $n$.

For $\alpha, \beta \in K^\times$, we set
\begin{align}\label{eq1}
(\alpha, \beta) := (\chi_\alpha, \beta) \in \text{Br}(K),
\end{align}
where $ (\chi_\alpha, \beta)$ is defined in Subsection \ref{subsec-local-symbols-chi-beta}.

\subsection{Local Symbols $(\alpha, \beta)_\upsilon$}
\label{subsec-local-symbols}

We recall from Serre \cite{serre-local-fields} how to define local symbols $(\alpha, \beta)_\upsilon$ for $\alpha, \beta \in K^\times$, where $K$ is a field that is complete under a discrete valuation $\upsilon$ with quasi-finite residue field $k$ such that $K$ contains the $n$th roots of unity. Note that $K$ contains the $n$th roots of unity if and only if $k$ contains the $n$th roots of unity (see Serre \cite[Lemma 1, p. 210]{serre-local-fields}). We further assume that $k$ has characteristic $0$ or $n$ is an integer prime to the characteristic of $k$.

Recall that a field $k$ is a quasi-finite field if the following two conditions hold:
\begin{itemize}

\item[(i)] $k$ is perfect;

\item[(ii)] there is an isomorphism of topological groups $\widehat{\bZ} \to G_k = \Gal (k^\sep / F)$ that is given by $\epsilon \in \widehat{\bZ} \mapsto \Xi^\epsilon \in G_k$ for some element $\Xi \in G_k$, where $\widehat{\bZ}=\varprojlim \bZ/ n\bZ$ is the profinite completion of $\bZ$, and $k^\sep$ is a separable closure of $k$.

\end{itemize}

Recall from Subsection \ref{subsec-symbol-alpha-beta} that for $\alpha, \beta \in K^\times$, one can define the symbol $(\alpha, \beta) = (\chi_\alpha, \beta)$ that belongs in $\text{Br}(K)$.

Fix a primitive $n$-th root of unity $\zeta_n$ that has the effect of identifying $\mu_n$ with $\bZ/ n\bZ$. We are now ready to define the local symbols that we will need in this paper.

\begin{definition} (Local Symbols)
\label{def-local-symbols}

For $\alpha, \beta \in K^\times$, set
\begin{align*}
(\alpha, \beta)_\upsilon := \zeta_n^{n \cdot \inv_K(\alpha, \beta)}
\end{align*}
where $(\alpha, \beta) \in \text{Br}(K)$ is the symbol in Subsection \ref{subsec-local-symbols}, and $n \cdot \inv_K(\alpha, \beta)$ is an element of $\bZ/ n\bZ$ since $\inv_K(\alpha, \beta)$ is an element of the subgroup $(1/n)\bZ/\bZ$ of $\bQ/\bZ$.

\end{definition}

\begin{remark}
\hfill
\begin{itemize}

\item[(i)] Local symbols $(\alpha, \beta)_\upsilon$ is an $n$-th root of unity.

\item[(ii)] By \cite[Corollary, p.208]{serre-local-fields}, the local symbols $(\alpha, \beta)_\upsilon$ do not depend on the choice of the root of unity $\zeta_n$.

\end{itemize}
\end{remark}

 We recall some basic properties of local symbols $(\alpha, \beta)_\upsilon$ from Serre \cite[Proposition 7, p.208]{serre-local-fields}.

\begin{proposition}
\label{prop-basic-properties-local-symbols}

\begin{itemize}

\item []

\item [(i)] $(\alpha_1\alpha_2, \beta)_\upsilon = (\alpha_1, \beta)_\upsilon(\alpha_2, \beta)_\upsilon$.

\item [(i)] $(\alpha, \beta_1\beta_2)_\upsilon = (\alpha, \beta_1)_\upsilon(\alpha, \beta_2)_\upsilon$.

\item [(iii)] $(\alpha, \beta)_\upsilon = 1$ if and only if $\beta$ is a norm in the extension $K(\alpha^{1/n})/K$.

\item [(iv)] $(\alpha, -\alpha)_\upsilon = (\alpha, 1 - \alpha)_{\upsilon} = 1$.

\item [(v)] $(\alpha, \beta)_{\upsilon} (\beta, \alpha)_\upsilon = 1$.

\item [(vi)] If $(\alpha, \beta)_{\upsilon} = 1$ for all $\beta \in K^{\times}$, then $\alpha \in K^{\times n}$.

\end{itemize}

\end{proposition}

Let $\cF$ be a topological generator of $G_k = \Gal(k^{\sep}/k)$. Given $x \in k^{\times}$, let $y \in k^{\sep^\times}$ be a solution to the equation $y^n = x$. Then $z = \cF(y)/y$ is an element in $\mu_n$, and does not depend on the choice of solution $y$. If we set $\cP_n(x) = \cF(y)/y$, then the map $\cP_n$ defines by passage to the quotient an isomorphism of $k^{\times}/k^{\sep^{\times}}$ onto $\mu_n$.

We recall the formula for computing $(\alpha, \beta)_{\upsilon}$ in Serre \cite[Proposition 8, p.210]{serre-local-fields}.

\begin{theorem}
\label{thm-computation-of-local-symbols-serre}

Let $\alpha$ and $\beta$ be elements in $K^{\times}$. Set
\begin{align*}
\gamma = (-1)^{\upsilon(\alpha)\upsilon(\beta)} \dfrac{\alpha^{\upsilon(\beta)}}{\beta^{\upsilon(\alpha)}}.
\end{align*}
Then $\gamma$ is a root of unity in $K$. Furthermore, if $\overline{\gamma}$ denotes its image in $k^{\times}$, then
\begin{align*}
(\alpha, \beta)_{\upsilon} = \cP_n(\overline{\gamma}).
\end{align*}

\end{theorem}

\subsubsection{\textbf{A special case where $k$ is an ultra-finite field}}
\label{subsec-local-symbols-when-k-is-an-ultra-finite-field}

In this subsection, we keep $K$ as in the beginning of Subsection \ref{subsec-local-symbols}, but further assume that its residue field is an ultra-finite residue field $\GF(\kappa)$ such that $\GF(\kappa)$ contains the $n$th roots of unity (see below for the notation regarding ultra-finite fields and their definition.) We are often concerned with the case where $K$ is the $P$-adic completion $\bK_P$ of the rational function field $\GF(\kappa)(t)$, where $P$ is either a monic irreducible polynomial of positive degree in the polynomial ring $\GF(\kappa)[t]$, or is the infinite prime $\infty = 1/t$. We are most interested in local symbols $(\alpha, \beta)_\upsilon$ when $n = 2$, and so in which case, $(\alpha, \beta)_\upsilon$ takes only values $\pm 1$.

Let $S$ be an infinite set, and fix a nonprincipal ultrafilter $\cD$ on $S$ as in Secion \ref{sec-basic-notions}. Throughout this paper, for each $s \in S$, let $q_s$ be a power of a prime $p_s$. We assume that the sequence $\{q_s\}_{s \in S}$ is unbounded for $\cD$-almost all $s \in S$, and thus the finite fields $\bF_{q_s}$ are of unbounded cardinalities. Let $\kappa = \ulim_{s \in S}q_s \in \bN^{\#}$. Recall from Nguyen \cite{nguyen-ultrafinite-fields-2023} that the \textbf{ultra-finite field $\GF(\kappa)$} is the ultraproduct of finite fields $\bF_{q_s}$, that is, $\GF(\kappa) = \prod_{s \in S}\bF_{q_s}/\cD$. Note that $\kappa$ is the internal cardinality of $\GF(\kappa)$. Since the finite fields $\bF_{q_s}$ are perfect, it follows that $\GF(\kappa)$ is also a perfect field (see Nguyen \cite{nguyen-ultrafinite-fields-2023}).

We first describe a topological generator $\cF$ of the absoulte Galois group $G_{\GF(\kappa)} = \Gal(\GF(\kappa)^{\alg}/\GF(\kappa))$ from Nguyen \cite{nguyen-ultrafinite-fields-2023}. For each $s \in S$, let $\cF_s$ denote the Frobenius map that sends each $x$ to $x^{q_s}$. For each positive integer $d > 0$, we know from Nguyen \cite{nguyen-ultrafinite-fields-2023} that $\GF(\kappa)$ has a unique extension of degree $d$ over $\GF(\kappa)$ that is of the form $\GF(\kappa^d) = \prod_{s\in S}\bF_{q_s^d}/\cD$. we let $\cF$ be the \textit{ultra-Frobenius map} that sends each element $a = \ulim_{s \in S}a_s \in \GF(\kappa^d)$ with $a_s \in \bF_{q_s^d}$ to $a^\kappa = \ulim_{s \in S}a_s^{q_s} \in \GF(\kappa^d)$.

\begin{lemma}
\label{lem-ultra-frobenius-is-a-generator}

$\GF(\kappa)$ is a quasi-finite field and the ultra-Frobenius map $\cF$ is a topological generator of the absolute Galois group $G_{\GF(\kappa)} = \Gal(\GF(\kappa)^{\alg}/\GF(\kappa))$.

\end{lemma}

\begin{proof}

We know from Nguyen \cite[Corollary 3.27 or Theorem 3.35]{nguyen-ultrafinite-fields-2023} that $\GF(\kappa)$ is a quasi-finite field. We now prove that $\cF$ is a topological generator of the absolute Galois group $G_{\GF(\kappa)}$. It suffices to prove that for each positive integer $d$, $\cF$ is a generator of the cyclic group $\Gal(H/\GF(\kappa))$ for a unique extension $H$ of $\GF(\kappa)$ of degree $d$. Recall from Nguyen \cite[Theorem 3.35]{nguyen-ultrafinite-fields-2023} that a unique extension of $\GF(\kappa)$ of degree $d$ is the ultra-finite field $\GF(\kappa^d) = \prod_{s \in S}\bF_{q_s^d}/\cD$.

Take an arbitrary irreducible polynomial $f$ of degree $d$ over $\GF(\kappa)$. Then by Nguyen \cite[Theorem 3.42]{nguyen-ultrafinite-fields-2023}, $f$ has a root $a$ in $\GF(\kappa^d)$. Furthermore, all the roots of $f$ are simple, and are exactly distinct elements $a$, $a^{\kappa}$, $a^{2\kappa}, \ldots, a^{(d - 1)\kappa}$ of $\GF(\kappa^d)$. Note that for each $0 \le i \le d - 1$,
\begin{align*}
\cF^i(a) = \underbrace{\cF \circ \cF \circ \cdots \circ \cF}_{\text{$i$ copies of $\cF$}}(a) = a^{i\kappa},
\end{align*}
and thus $1 = \cF^0, \cF, \cdots, \cF^{d - 1}$ are $d$ distinct maps. Thus $\cF$ generates the cyclic group $\Gal(\GF(\kappa^d)/\GF(\kappa))$ of order $d$. Therefore $\cF$ is a topological generator of the absolute Galois group $G_{\GF(\kappa)}$.

\end{proof}

Using the above lemma, we can explicitly describe the isomorphism $\cP_n$ in Theorem \ref{thm-computation-of-local-symbols-serre}. In order that $\GF(\kappa)$ contains the $n$th roots of unity $\mu_n$, we know from Nguyen \cite[Example 4.10]{nguyen-ultrafinite-fields-2023} that it is necessary and sufficient that $n$ divides $\kappa - 1$ in $\bZ^{\#}$.

Suppose that $n$ is a positive integer that divides $\kappa - 1$ in $\bZ^{\#}$.

Given $x \in \GF(\kappa)^{\times}$, let $y \in \GF(\kappa)^{\alg^\times}$ be a solution to the equation $y^n = x$. Thus
\begin{align*}
\cP_n(x) = \cF(y)/y = y^{\kappa}/y = y^{\kappa - 1} = x^{\frac{\kappa - 1}{n}}.
\end{align*}

The following follows immediately from Theorem \ref{thm-computation-of-local-symbols-serre}.

\begin{corollary}
\label{cor-computation-of-local-symbols-nguyen}

Let $\alpha$ and $\beta$ be elements in $K^{\times}$. Set
\begin{align*}
\gamma = (-1)^{\upsilon(\alpha)\upsilon(\beta)} \dfrac{\alpha^{\upsilon(\beta)}}{\beta^{\upsilon(\alpha)}}.
\end{align*}
Then $\gamma$ is a root of unity in $K$. Furthermore, if $\overline{\gamma}$ denotes its image in $\GF(\kappa)^{\times}$, then
\begin{align*}
(\alpha, \beta)_{\upsilon} = \cP_n(\overline{\gamma}) = \overline{\gamma}^{\frac{\kappa - 1}{n}}.
\end{align*}

\end{corollary}

\subsection{Quadratic residue symbol in $\GF(\kappa)[t]$}
\label{subsec-quadratic-residue-symbol}

For the rest of this paper, we let $\bA = \GF(\kappa)[t]$ be the polynomial ring over the ultra-finite field $\GF(\kappa)$, where $\kappa = \ulim_{s \in S}q_s$ as in Subsection \ref{subsec-local-symbols-when-k-is-an-ultra-finite-field}. We also fix the notation $\bK$ for the rational field $\GF(\kappa)(t)$ which is the field of fractions of $\bA$.

Suppose that $n$ is a positive integer that divides $\kappa - 1$ in $\bZ^{\#}$. This condition is equivalent to that $\GF(\kappa)$ contains the $n$-th roots of unity $\mu_n$ (see Nguyen \cite[Example 4.10]{nguyen-ultrafinite-fields-2023}). For each $s \in S$, let $\bA_s = \bF_{q_s}[t]$ be the polynomial ring over $\bF_{q_s}$. Let $\cU(\bA)$ be the ultrahull of $\bA$, i.e., $\cU(\bA) = \prod_{s \in S}\bA_s/\cD$. Let $P$ be an irreducible polynomial of degree $d \in \bZ_{>0}$ in $\bA$, and let $\alpha$ be a polynomial in $\bA$ such that $\alpha, P$ are relatively prime in $\bA$. We recall from Nguyen \cite[Definition 4.12]{nguyen-ultrafinite-fields-2023} that the $n$th power residue symbol $\left(\dfrac{\alpha}{P}\right)$ is the unique element in the subgroup $\mu_n$ of the $n$-th roots of unity of $\GF(\kappa)^{\times}$ such that
\begin{align}
\label{def-nth-power-residue-symbol}
\alpha^{\frac{\kappa^d - 1}{n}} \equiv \left(\dfrac{\alpha}{P}\right) \pmod{P\cU(\bA)}.
\end{align}

If $P$ divides $\alpha$, we simply define $\left(\dfrac{\alpha}{P}\right) = 0$.

We recall some basic properties of the $n$-th power residue symbol (see Nguyen \cite[Proposition 4.15]{nguyen-ultrafinite-fields-2023}).

\begin{proposition}
\label{prop-properties-of-power-residue-symbol}

The $n$-th power residue symbol has the following properties: for any $\alpha, \lambda \in \bA$,
\begin{itemize}

\item[(i)] $\left( \dfrac{\alpha}{P} \right)_n = \left( \dfrac{\lambda}{P} \right)_n$ iff $\alpha \equiv \lambda \pmod{P\bA}$.

\item[(ii)] $\left( \dfrac{\alpha\lambda}{P} \right)_n = \left( \dfrac{\alpha}{P} \right)_n \left( \dfrac{\lambda}{P}\right)_n$.

\item[(iii)] $\left( \dfrac{\alpha}{P} \right)_n = 1$ if and only if the equation $x^n \equiv \alpha \pmod{P\cU(\bA)} $ is solvable in $\bA$. If $n$ is a positive integer, then $\left( \dfrac{\alpha}{P} \right)_n = 1$ if and only if the equation $x^n \equiv \alpha \pmod{P\bA}$ is solvable.

\item[(iv)] Let $\zeta$ be an element in $\GF(\kappa)^\times$ of order dividing $n$. Then there exists an element $\alpha \in \bA$ such that
\begin{align*}
\left( \dfrac{\alpha}{P} \right)_n  = \zeta.
\end{align*}

\end{itemize}

\end{proposition}

We can extend the notion of the $n$-th power residue symbol to the case where the prime $P$ is replaced by an arbitrary nonzero element in $\bA$ of positive degree. For a nonzero polynomial $\beta \in \bA$, write $\beta = aP_1^{d_1}\cdots P_m^{d_m}$ for the prime factorization of $\beta$ into monic irreducible polynomials $P_i$ in $\bA$, where $a \in \GF(\kappa)^{\times}$ is the leading coefficient of $\beta$ and the $d_i$ are positive integers. The constant $a$ is called the \textbf{sign of $\beta$} which we subsequently denote by $a = \sign(\beta)$. Define
\begin{align}
\label{def-general-nth-power-residue-symbol}
\left(\dfrac{\alpha}{\beta}\right)_n := \prod_{i = 1}^m  \left(\dfrac{\alpha}{P_i}\right)^{d_i}_n \in \mu_n \subseteq \GF(\kappa)^{\times}.
\end{align}

For each polynomial $f \in \bA$, let  $a = \ulim_{s \in S}a_s \in \GF(\kappa)^{\times}$ be the leading coefficient of $f$, where the $a_s$ belong in $\bF_{q_s}^{\times}$. Set
\begin{align}
\label{def-signature-of-kappa-1/n-ultra-power}
\text{sign}_n(f) = a^{\frac{\kappa - 1}{n}} = \ulim_{s\in S}a_s^{\frac{q_s - 1}{n}} \in \GF(\kappa)^{\times}.
\end{align}

We recall the general higher reciprocity law that is proved in Nguyen \cite[Theorem 5.7]{nguyen-ultrafinite-fields-2023}.

\begin{theorem}
(The general reciprocity law)
\label{thm-General-Reciprocity-Law}

Let $\alpha, \beta$ be relatively prime nonzero elements in $\bA$. Then
\begin{align*}
\left( \dfrac{\alpha}{\beta} \right)_n \left(\dfrac{\beta}{\alpha} \right)_n^{-1} = (-1)^{\frac{\kappa-1}{n} \deg(\alpha)\deg(\beta)} \text{sign}_n(\alpha)^{\deg(\beta)}  \text{sign}_n(\beta)^{-\deg(\alpha)}.
\end{align*}

\end{theorem}

Whenever $n = 2$, we suppress the subscript $2$ in the quadratic residue symbol, and thus denote by $\left( \dfrac{\cdot}{\cdot} \right)$ for the quadratic residue symbol.

\begin{remark}
\label{rem-fixed-notation-about-local-symbols}

For the rest of the paper, unless otherwise stated, we only consider local symbols $(\alpha, \beta)_{\upsilon}$ in the $P$-adic completion $\bK_P$ of the rational function field $\bK := \GF(\kappa)(t)$ that corresponds to $n = 2$, where $\upsilon$ is the $P$-adic valuation $v_P$. By abuse of notation, we also write $(\alpha, \beta)_P$ for the local symbols $(\alpha, \beta)_{v_P}$. Thus these local symbols only take values $\pm 1$, and can be used to verify the ramification of certain quaternion algebras as follows. For $a, b \in \bK$, the quaternion algebra  $\bH_{a, b} = \bK \cdot 1 \bigoplus \bK \cdot \alpha \bigoplus \bK \cdot \beta \bigoplus \bK \cdot \alpha\beta$, where $\alpha^2 = a$ and $\beta^2 = b$ such that $\alpha\beta = -\beta\alpha$, is split (or unramified) at a prime $P$, i.e., $\bH_{a, b} \otimes \bK_P$ is isomorphic to the algebra of $2\times 2$ matrices over $\bK_P$ if and only $(a, b)_P = 1$, and $\bH_{a, b}$ is nonsplit (or ramified) at $P$, i.e., $\bH_{a, b} \otimes \bK_P$ is a division algebra over $\bK_P$ if and only if $(a, b)_P = -1$.

\end{remark}

\section{A definition of $\bA$ in $\bK$}
\label{sec-definition-of-A-in-K}

We let $\bA = \GF(\kappa)[t]$ and $\bK = \GF(\kappa)(t)$ as in the last section. We further fix the following notation.
\begin{description}[style=multiline, labelwidth=1.5cm]

\item [\namedlabel{itm: N1}{(N1)}] $\bP = \{P \in \bA \; : \; \text{$P$ is a monic irreducible polynomial of positive degree}\}$, and let $\infty = 1/t$ denote the infinite prime of $\bA$.

\item [\namedlabel{itm: N2}{(N2)}] For a given prime $P \in \bP \cup \{\infty\}$, let $\bK_P$ denotes the $P$-adic completion of $\bK$ with respect to the standard $P$-adic valuation $v_P$ of $\bK$. The $\infty$-adic valuation corresponds to the valuation $-\deg$ in which case $\bK_{\infty} = \GF(\kappa)((1/t))$. Let $\cO_{\bK_P}$ denote the valuation ring of $\bK_P$ that consists of all elements $\alpha$ in $\bK_P$ such that $v_P(\alpha) \ge 0$.

    The valuation ring of $\bK$ with respect to $v_P$ will be denoted by $\bA_P$ that can be written in the form
    \begin{align*}
    \bA_P := \{\alpha \in \bK = \GF(\kappa)(t) \; | \; v_P(\alpha) \ge 0 \}.
    \end{align*}

    For every non-infinite prime $P \in \bP$, we know from Bourbaki \cite[Corollary 2, p.380]{bourbaki-CA} that $\bA_P$ is the localization of $\bA$ at prime $P$, that consists exactly of the fractions $\dfrac{\alpha}{\beta} \in \bK$, where $\alpha, \beta$ are relatively prime polynomials in $\bA$ such that $P$ does not divide $\beta$.

    Also by Bourbaki \cite[Corollary 2, p.380]{bourbaki-CA}, the valuation ring $\bA_{\infty}$ is the localization of the ring $\GF(\kappa)[1/t]$ at the prime ideal $(1/t)\GF(\kappa)[1/t]$; in other words, $\bA_{\infty}$ consists exactly of the fractions $\dfrac{\alpha}{\beta}$, where $\alpha, \beta \in \bA$ such that $\deg(\alpha) \le \deg(\beta)$.

    The residue field of $\bK_P$ will denoted by $\GF(\kappa)_P$ which is isomorphic to either the field $\bA/P\bA$ whenever $P \in \bP$ or $\GF(\kappa)$ whenever $P = \infty$. Note that the field $\bA/P\bA$ is the unique extension field of degree $\deg(P)$ over $\GF(\kappa)$ which is also an ultra-finite field. Following the notation in Nguyen \cite{nguyen-ultrafinite-fields-2023}, $\bA/P\bA$ in fact is the ultra-finite field $\GF(\kappa^d) = \prod_{s \in S}\bF_{q_s^d}/\cD$, where $d = \deg(P)$. The corresponding residue map $\cO_{\bK_P} \to \GF(\kappa)_P$ is denoted by $\red_P$.

    For an arbitrary prime $Q \in \bA$, let $P = Q/a$, where $a$ is the leading coefficient of $Q$. Thus $P$ is a monic prime in $\bP$. The $Q$-adic valuation is exactly the $P$-adic valuation, and thus we will also use notation $v_Q, \bK_Q, \bA_Q, \red_Q$ to denote $v_P, \bK_P, \bA_P, \red_P$, respectively.

\item [\namedlabel{itm: N3}{(N3)}] for elements $a, b \in \bK$, let $\bH_{a, b}$ is the quaternion algebra $\bK \cdot 1 \bigoplus \bK \cdot \alpha \bigoplus \bK \cdot \beta \bigoplus \bK \cdot \alpha\beta$, where $\alpha^2 = a$ and $\beta^2 = b$ such that $\alpha\beta = -\beta\alpha$.

\item [\namedlabel{itm: N4}{(N4)}] for a given element $a \in \bK$, let
\begin{align*}
\bP(a) = \{P \in \bP \cup \{\infty\} \; : \;  \text{$v_P(a)$ is odd}\}.
\end{align*}

\item [\namedlabel{itm: N5}{(N5)}] for elements $a, b \in \bK^{\times}$, let
\begin{align*}
\Delta_{a, b} = \{P \in \bP \cup \{\infty\} \; : \; \text{$\bH_{a, b}$ is nonsplit (or ramified) at $P$}\}.
\end{align*}

\item [\namedlabel{itm: N6}{(N6)}] for elements $a, b \in \bK$, let
\begin{align*}
\widetilde{\cR_{a, b}} := \bigcup_{P \in \Delta_{a, b}\cap (\bP(a) \cup \bP(b))}\bA_P,
\end{align*}

\end{description}

Using notation \ref{itm: N2}, the following result is immediate from Bourbaki \cite[Theorem 3, p.378]{bourbaki-CA}, and will be useful in the proof of Theorem \ref{thm-main-thm1} below.

\begin{lemma}
\label{lem-intersection-of-valuation-rings}

We have
\begin{align*}
\bA = \bigcap_{P \in \bP, P \ne \infty}\bA_P.
\end{align*}

\end{lemma}

\begin{remark}
\label{rem-residue-symbol-equal-reduction-to-a-power}

For a given monic prime $P \in \bA$, since $\bA/P\bA \cong \cU(\bA)/P\cU(\bA)$ (see Nguyen \cite{nguyen-ultrafinite-fields-2023} or \cite{schoutens}), we obtain the following commutative diagram
\begin{center}
\begin{tikzcd}
  \bA \arrow[r, hook]{r}{\iota} \arrow[dr, rightarrow]{dr}{\red_P}
  & \cU(\bA) \arrow[d] \\
  & \bA/P\bA \cong \cU(\bA)/P\cU(\bA)
\end{tikzcd}
\end{center}
where $\iota : \bA \hookrightarrow \cU(\bA)$ is an embedding, and $\red_P$ denotes the restriction of the residue map to $\bA$. Thus we deduce from the diagram that for any polynomial $\alpha \in \bA$,
\begin{align*}
\alpha \equiv \red_P(\alpha) \pmod{P\cU(\bA)},
\end{align*}
where we identify $\bA/P\bA \cong \cU(\bA)/P\cU(\bA)$ with a unique extension $\GF(\kappa^{\deg(P)})$ of degree $\deg(P)$ over $\GF(\kappa)$. It follows from (\ref{def-nth-power-residue-symbol}) that
\begin{align}
\label{e-quadratic-residue-symbol-equal-reduction-to-a-power}
\left(\dfrac{\alpha}{P}\right)_n = \red_P(\alpha)^{\frac{\kappa^{\deg(P)} - 1}{n}}.
\end{align}

\end{remark}

Using notation \ref{itm: N2} above and replacing $\GF(\kappa)$ by $\GF(\kappa^{\deg(P)})$ as the residue field in Corollary \ref{cor-computation-of-local-symbols-nguyen}, we immediately obtain the following.

\begin{theorem}
\label{thm-Nguyen-formula-for-local-symbols-K_P}

Let $\alpha$ and $\beta$ be elements in $\bK_P^{\times}$. \begin{itemize}

\item [(i)] if $P$ is a monic irreducible polynomial in $\bA$ of positive degree, then
\begin{align*}
(\alpha, \beta)_P = \left((-1)^{v_P(\alpha)v_P(\beta)} \red_P\left(\dfrac{\alpha^{v_P(\beta)}}{\beta^{v_P(\alpha)}} \right)\right)^{\frac{\kappa^{\deg(P)} - 1}{n}}.
\end{align*}

\item [(ii)] if $P$ is the infinite prime $\infty = 1/t$, then
\begin{align*}
(\alpha, \beta)_{\infty} = \left((-1)^{v_{\infty}(\alpha)v_{\infty}(\beta)} \red_{\infty}\left(\dfrac{\alpha^{v_{\infty}(\beta)}}{\beta^{v_{\infty}(\alpha)}} \right)\right)^{\frac{\kappa - 1}{n}}.
\end{align*}

\end{itemize}

\end{theorem}

From the above theorem and the general reciprocity law \ref{thm-General-Reciprocity-Law}, we deduce an analogue of the Hilbert reciprocity law for the rational function field $\bK = \GF(\kappa)(t)$.

\begin{theorem}
\label{thm-analogue-of-the-Hilbert-reciprocity-law}
(Analogue of the Hilbert reciprocity law for $\bK = \GF(\kappa)(t)$)

For any nonzero elemenets $\alpha, \beta \in \bK$,
\begin{align*}
\prod_{P \in \bP \cup \{\infty\}}(\alpha, \beta)_P = 1.
\end{align*}

\end{theorem}

\begin{proof}

Since $\alpha, \beta$ are nonzero elements in $\bK$, we know from Theorem \ref{thm-Nguyen-formula-for-local-symbols-K_P} that $(\alpha, \beta)_P = 1$ for all but finitely many primes $P \in \bP\cup \{\infty\}$. On the other hand, by Proposition \ref{prop-basic-properties-local-symbols}, the local symbol $(\cdot, \cdot)_P$, for a given prime $P$, is multiplicative in both components, and thus it suffices to consider the case where both $\alpha, \beta$ are (not necessarily monic) primes in $\bK$. Write $\alpha = aP$ and $\beta = bQ$ for some monic primes $P, Q$, where $a, b \in \GF(\kappa)^{\times}$ are the leading coefficients of $\alpha, \beta$, respectively.

It is easy to see from Theorem \ref{thm-Nguyen-formula-for-local-symbols-K_P} that if $\fp$ is an arbitrary monic prime in $\bK$ such that $\fp \ne P, Q, \infty$, then $(\alpha, \beta)_{\fp} = 1$ since both $v_{\fp}(\alpha)$ and $v_{\fp}(\beta)$ are zero. Thus
\begin{align}
\label{e-eqn1-Hilbert-reciprocity-law-analogue}
\prod_{\fp \in \bP \cup \{\infty\}}(\alpha, \beta)_{\fp} = (\alpha, \beta)_P(\alpha, \beta)_Q(\alpha, \beta)_{\infty}.
\end{align}

Since $v_P(\alpha) = 1$ and $v_P(\beta) = 0$, we know from (\ref{e-quadratic-residue-symbol-equal-reduction-to-a-power}) and Theorem \ref{thm-Nguyen-formula-for-local-symbols-K_P} that
\begin{align}
\label{e-eqn2-Hilbert-reciprocity-law-analogue}
(\alpha, \beta)_P = \red_P(\beta^{-1})^{\frac{\kappa^{\deg(P)} - 1}{n}} = \left(\dfrac{\beta}{P}\right)_n^{-1} = \left(\dfrac{\beta}{\alpha}\right)_n^{-1}.
\end{align}
Similarly, we can show that
\begin{align}
\label{e-eqn3-Hilbert-reciprocity-law-analogue}
(\alpha, \beta)_Q = \red_Q(\alpha)^{\frac{\kappa^{\deg(Q)} - 1}{n}} = \left(\dfrac{\alpha}{Q}\right)_n= \left(\dfrac{\alpha}{\beta}\right)_n.
\end{align}

Let
\begin{align*}
P = t^{\deg(P)} + \sum_{i = 0}^{\deg(P)- 1}p_it^i,
\end{align*}
and
\begin{align*}
Q = t^{\deg(Q)} + \sum_{i = 0}^{\deg(Q)- 1}q_it^i.
\end{align*}
In dividing $\frac{Q^{\deg(P)}}{P^{\deg(Q)}}$ by $t^{\deg(P)\deg(Q)}$, we obtain
\begin{align*}
\dfrac{Q^{\deg(P)}}{P^{\deg(Q)}} = \dfrac{\left(1 + \sum_{i = 0}^{\deg(Q) - 1}q_i(1/t)^{\deg(Q) - i}\right)^{\deg(P)}}{\left(1 + \sum_{i = 0}^{\deg(P) - 1}p_i(1/t)^{\deg(P) - i}\right)^{\deg(Q)}} \equiv 1 \pmod{\infty},
\end{align*}
and thus
\begin{align}
\label{e-eqn4-Hilbert-reciprocity-law}
\red_{\infty}\left(\dfrac{Q^{\deg(P)}}{P^{\deg(Q)}}\right) = 1.
\end{align}

Since $v_\infty(\alpha) = -\deg(\alpha) = -\deg(P)$, $v_{\infty}(\beta) = -\deg(\beta) = -\deg(Q)$, and the residue field of $\bK_{\infty}$ is $\GF(\kappa)$, Theorem \ref{thm-Nguyen-formula-for-local-symbols-K_P} implies that
\begin{align*}
(\alpha, \beta)_{\infty} &= (-1)^{\deg(\alpha)\deg(\beta)\frac{\kappa - 1}{n}}\left(\red_{\infty}\left(\dfrac{\alpha^{v_{\infty}(\beta)}}{\beta^{v_{\infty}(\alpha)}}\right)\right)^{\frac{\kappa - 1}{n}} \\
&= (-1)^{\deg(\alpha)\deg(\beta)\frac{\kappa - 1}{n}}\left(\red_{\infty}\left(\dfrac{a^{-\deg(Q)}P^{-\deg(Q)}}{b^{-\deg(P)}Q^{-\deg(P)}}\right)\right)^{\frac{\kappa - 1}{n}} \\
&= (-1)^{\deg(\alpha)\deg(\beta)\frac{\kappa - 1}{n}}\red_{\infty}\left(\dfrac{a^{-\deg(Q)}}{b^{-\deg(P)}}\right)^{\frac{\kappa - 1}{n}} \; \; (\;by \; (\ref{e-eqn4-Hilbert-reciprocity-law})). \; \\
\end{align*}

Since $a, b$ are nonzero elements in $\GF(\kappa)$, we know from (\ref{def-signature-of-kappa-1/n-ultra-power}) that
\begin{align*}
\red_{\infty}\left(\dfrac{a^{-\deg(Q)}}{b^{-\deg(P)}}\right)^{\frac{\kappa - 1}{n}} = \sign_n(\alpha)^{-\deg(\beta)}\sign_n(\beta)^{\deg(\alpha)},
\end{align*}
and therefore
\begin{align}
\label{e-eqn5-Hilbert-reciprocity-law}
(\alpha, \beta)_{\infty} = (-1)^{\deg(\alpha)\deg(\beta)\frac{\kappa - 1}{n}}\sign_n(\alpha)^{-\deg(\beta)}\sign_n(\beta)^{\deg(\alpha)}.
\end{align}

By (\ref{e-eqn1-Hilbert-reciprocity-law-analogue}), (\ref{e-eqn2-Hilbert-reciprocity-law-analogue}), (\ref{e-eqn3-Hilbert-reciprocity-law-analogue}), (\ref{e-eqn5-Hilbert-reciprocity-law}), and the general reciprocity law \ref{thm-General-Reciprocity-Law}, we deduce that
\begin{align*}
\prod_{\fp \in \bP \cup \{\infty\}}(\alpha, \beta)_{\fp} = \left(\dfrac{\beta}{\alpha}\right)_n^{-1}\left(\dfrac{\alpha}{\beta}\right)_n(-1)^{\deg(\alpha)\deg(\beta)\frac{\kappa - 1}{n}}\sign_n(\alpha)^{-\deg(\beta)}\sign_n(\beta)^{\deg(\alpha)} = 1,
\end{align*}
which proves the theorem.

\end{proof}

When $n = 2$, the definition of $\kappa$ trivially implies that $2$ divides $\kappa - 1$, and thus $2$ divides $\kappa^{\deg(P)} - 1$. So the above theorem implies the following formula for local symbols corresponding to $n = 2$ over the $P$-adic completion $\bK_P$.
\begin{corollary}
\label{cor-Nguyen-formula-for-local-symbols-K_P-quadratic-symbol}

Let $\alpha$ and $\beta$ be elements in $\bK_P^{\times}$. Then
\begin{itemize}

\item [(i)] if $P$ is a monic irreducible polynomial in $\bA$ of positive degree, then
\begin{align*}
(\alpha, \beta)_P = \left((-1)^{v_P(\alpha)v_P(\beta)} \red_P\left(\dfrac{\alpha^{v_P(\beta)}}{\beta^{v_P(\alpha)}} \right)\right)^{\frac{\kappa^{\deg(P)} - 1}{2}}.
\end{align*}

\item [(ii)] if $P$ is the infinite prime $\infty = 1/t$, then
\begin{align*}
(\alpha, \beta)_{\infty} = \left((-1)^{v_{\infty}(\alpha)v_{\infty}(\beta)} \red_{\infty}\left(\dfrac{\alpha^{v_{\infty}(\beta)}}{\beta^{v_{\infty}(\alpha)}} \right)\right)^{\frac{\kappa - 1}{2}}.
\end{align*}

\end{itemize}

\end{corollary}

\subsection{Description of $\Delta_{a, b}$}
\label{subsec-Delta-a-b}

In this subsection, we describe the sets $\Delta_{a, b}$ in \ref{itm: N5} in more detail.

For a prime $P \in \bP \cup \{\infty\}$, define
\begin{align*}
h =
\begin{cases}
\deg(P) \; \; &\text{if $P$ is a non-infinite prime in $\bP$},\\
1 \; \; &\text{if $P =\infty$}.
\end{cases}
\end{align*}

By Proposition \ref{prop-basic-properties-local-symbols}(iii), $P \in \Delta_{a, b}$ if and only if $(a, b)_P = -1$ if and only if the following is true.
\begin{align}
\label{e-eqn1-in-main-thm1}
\left((-1)^{v_P(a)v_P(b)} \red_P\left(\dfrac{(a)^{v_P(b)}}{(b)^{v_P(a))}}\right)\right)^{\frac{\kappa^h - 1}{2}} = -1 \; \; (\text{by Corollary \ref{cor-Nguyen-formula-for-local-symbols-K_P-quadratic-symbol}}).
\end{align}

Suppose that $P \not\in \bP(a) \cup \bP(b)$.  (see notation \ref{itm: N4}). Then both $v_P(a)$ and $v_P(b)$ are even, and thus both $v_P(a)/2$ are $v_P(b)/2$ are integers. Setting $\alpha = \dfrac{a^{v_P(b)/2}}{b^{v_P(a)/2}}$, we see that $\alpha \in \bK$. Since $v_P(\alpha) = 0$, $\alpha$ is a unit in $\bA_P$, and thus $\red_P(\alpha) \in \GF(\kappa)_P^{\times}$. By Nguyen \cite[Theorem 3.35]{nguyen-ultrafinite-fields-2023}, $\GF(\kappa)_P = \bA/P\bA$ is a unique extension of degree $h$ over $\GF(\kappa)$, and can be written in the form
\begin{align*}
\GF(\kappa)_P = \GF(\kappa^h) = \prod_{s\in S}\bF_{q_s^h}/\cD,
\end{align*}
and thus $\red_P(\alpha) = \ulim_{s\in S}\alpha_s \in \GF(\kappa^h)^{\times}$ for some elements $\alpha_s \in \bF_{q_s^h}^{\times}$.

By the theory of finite fields, $\alpha_s^{q_s^h - 1} = 1$, and thus
\begin{align}
\label{e-eqn2-in-main-thm1}
\red_P(\alpha)^{\kappa^h - 1} = \ulim_{s\in S}\alpha_s^{q_s^h - 1} = 1.
\end{align}

On the other hand, since both $v_P(a)$ and $v_P(b)$ are even, we know from (\ref{e-eqn1-in-main-thm1}) that
\begin{align*}
\red_P\left(\dfrac{(a)^{v_P(b)}}{(b)^{v_P(a))}}\right)^{\frac{\kappa^{\deg(P)} - 1}{2}} = \red_P(\alpha^2)^{\frac{\kappa^h - 1}{2}} = \red_P(\alpha)^{\kappa^h - 1} = -1,
\end{align*}
which is a contradiction to (\ref{e-eqn2-in-main-thm1}). Thus $P$ must belong in $\bP(a) \cup \bP(b)$, and thus
\begin{align*}
\Delta_{a, b}\cap (\bP(a) \cup \bP(b)) = \Delta_{a, b}.
\end{align*}

We summarize the above discussion in the following.

\begin{proposition}
\label{pro-description-of-Delta-a-b}

For any elements $a, b \in \bK^{\times}$,
\begin{align*}
\Delta_{a, b}\cap (\bP(a) \cup \bP(b)) = \Delta_{a, b}.
\end{align*}
In particular, $\Delta_{a, b}$ is a finite set.

\end{proposition}

\begin{proof}

The first assertion is immediate from the discussion preceding the proposition. For the last assertion, we know that if a prime $P \in \bP \cup \{\infty\}$ belongs in $\Delta_{a, b}$, then $P$ belongs in $\bP(a) \cup \bP(b) $. Both $\bP(a)$ and $\bP(b)$ are finite sets, and thus $\Delta_{a, b}$ is a finite set.

\end{proof}

\subsection{Defining $\GF(\kappa)$ in $\bK = \GF(\kappa)(t)$.}
\label{subsec-koe-2002-defining-transcendentals}

Recall from Pop \cite{pop-1996} that a field $F$ is \textbf{large} if any $F$-variety of dimension at least $1$ with one smooth $F$-point has infinitely many $F$-rational points. Standard examples of large fields include algebraically closed fields, henselian valued fields, and pseudo-algebraically closed (PAC) fields (see \cite{FJ}). It is known that ultra-finite fields are PAC (see Ax \cite{ax-1968}), and thus $\GF(\kappa)$ is large.

The following theorem is due to Koenigsmann \cite{koe-2002}.

\begin{theorem}
\label{thm-Koenigsmann-2002}
(Koenigsmann \cite[Theorem 2]{koe-2002})

Let $F$ be a large field, and let $K/F$ be a function field in one variable. Then $F$ is definable in $K$ by an existential parameter-free formula in the language of fields.

\end{theorem}

The above theorem immediately implies the following.

\begin{corollary}
\label{cor-Koe-2002}

$\GF(\kappa)$ is definable in $\bK = \GF(\kappa)(t)$ by an existential parameter--free formula in the language of fields.

\end{corollary}

We recall the formula from Koenigsmann \cite{koe-2002} that defines $\GF(\kappa)$ in $\bK$. Let $\bF$ be the prime field of $\GF(\kappa)$, so $\bF$ is either the rationals $\bQ$ or a finite field $\bF_p$ for some prime number $p >0$. Let $f \in \bF[X, Y]$ define a plane curve $\cC$ over $\bF$ of genus $> 0$ that has a regular $\bF$-rational point, say $(a, b) \in \cC(\bF)$ such that $\frac{df}{dY}(a, b) \ne 0$. We denote by $\psi(x)$ the formula defined by
\begin{align}
\label{e-psi(x)-formula-in-koe-2002}
\exists x_1, x_2, y_1, y_2 \in \bK\; :\; x_1 = xx_2 \land x_2 \ne 0 \land f(a + x_1, y_1) = 0 \land f(a + x_2, y_2) = 0.
\end{align}
Then Koenigsmann \cite{koe-2002} proves that for any $x \in \bK$,
\begin{align}
\label{e-koe-formula-2002}
x\in \GF(\kappa) \Leftrightarrow \psi(x).
\end{align}

\subsection{The Jacobson radical $\cJ\left(\cT_{a, b}\right)$}
\label{subsec-Jacobson-radical-of-R_{a,b}}

In addition to notations \ref{itm: N1}--\ref{itm: N6}, we further fix the following notation for the rest of the paper.

\begin{description}[style=multiline, labelwidth=1.5cm]

\item [\namedlabel{itm: N7}{(N7)}] For any $a, b \in \bK^{\times}$, let
\begin{align*}
\cS_{a, b} = \{2x_1 \in \bK \; | \; \text{$\exists x_2, x_3, x_4 \in \bK$ such that $x_1^2 - ax_2^2 - bx_3^2 + abx_4^2 = 1$}  \}
\end{align*}
be the set of traces of norm-$1$ elements of the quaternion algebra $\bH_{a, b}$.

If $P$ is a prime in $\bP \cup \{\infty\}$, we will similarly denote by $\cS_{a, b}(\bK_P)$ the set of traces of norm-$1$ elements of the quaternion algebra $\bH_{a, b} \otimes \bK_P$, i.e., replacing $\bK$ by $\bK_P$ in the definition of $\cS_{a, b}$.

\item [\namedlabel{itm: N8}{(N8)}] $\cT_{a, b} := \cS_{a, b} + \cS_{a, b}$.

\item [\namedlabel{itm: N9}{(N9)}] for each non-infinite prime $P \in \bP$, define
\begin{align*}
\cU_P = \{\epsilon \in \GF(\kappa)_P \cong \bA/P\bA \; |\; \text{the polynomial $x^2 - \epsilon x + 1$ is irreducible over $\GF(\kappa)_P$}\}.
\end{align*}

Note that using Nguyen \cite[Theorem 3.35]{nguyen-ultrafinite-fields-2023}, $\GF(\kappa)_P$ is a unique extension $\GF(\kappa^{\deg(P)})$ of degree $\deg(P)$ over $\GF(\kappa)$ that can be written in the form $\prod_{s \in S}\bF_{q_s^{\deg(P)}}/\cD$.

Similarly, we define
\begin{align*}
\cU_{\infty} = \{\epsilon \in \GF(\kappa)_{\infty} = \GF(\kappa)  \; |\; \text{the polynomial $x^2 - \epsilon x + 1$ is irreducible over $\GF(\kappa)$}\}.
\end{align*}

\end{description}

For a prime power $q$ in $\bZ$, set
\begin{align*}
U_q = \{s \in \bF_q \; | \; \text{$x^2 - sx + 1$ is irreducible in $\bF_q[x]$}\}.
\end{align*}

The following result is due to Poonen \cite[Lemma 2.3]{poonen-2009}.
\begin{lemma}
\label{lem-the-set-U-q-poonen-lemma}

If $q$ is a prime power strictly greater than $11$, then $U_q + U_q = \bF_q$.

\end{lemma}

\begin{lemma}
\label{lem-cU-P-is-the-ultraproduct-of-U-q}

Let $P \in \bP \cup \{\infty\}$ be a prime of $\bA$. Then
\begin{itemize}

\item [(i)] if $P$ is a non-infinite prime of degree $d$ in $\bA$, then $\cU_P = \prod_{s \in S}U_{q_s^d}/\cD$ and $\cU_P + \cU_P = \GF(\kappa^d)$.

\item [(ii)] $\cU_{\infty} = \prod_{s \in S}U_{q_s}/\cD$ and $\cU_{\infty} + \cU_{\infty} = \GF(\kappa)$

\end{itemize}

\end{lemma}

\begin{proof}

We first prove part (i). By the remark in \ref{itm: N9}, the residue field $\GF(\kappa)_P = \GF(\kappa^d) = \prod_{s \in S}\bF_{q_s^d}/\cD$.

We prove the first assertion of part (i). Let $\epsilon = \ulim_{s \in S}\epsilon_s \in \cU_P$ for some $\epsilon_s \in \bF_{q_s^d}$. Then the polynomial $f(x) = x^2 - \epsilon x + 1$ is irreducible over $\GF(\kappa^d)$. We see that
\begin{align*}
f(x) = \ulim_{s \in S}f_s(x),
\end{align*}
where $f_s(x) = x^2 - \epsilon_s x + 1$ is a polynomial in $\bF_{q_s^d}[x]$ for each $s \in S$. Since irreducibility can be defined using a first-order formula in the language of fields, \L{}o\'s' theorem implies that $f_s(x)$ is irreducible over $\bF_{q_s^d}$ for $\cD$-almost all $s \in S$, and thus $\epsilon_s \in U_{q_s^d}$ for $\cD$-almost all $s \in S$. Therefore
\begin{align*}
\cU_P \subseteq \prod_{s\in S}U_{q_s^d}/\cD.
\end{align*}

The reverse inclusion follows exactly the same arguments as above, and thus the first assertion of part (i) follows immediately.

For the last assertion of part (i), there exists a finite subset $S_0 \subset S$ such that $q_s > 11$ for all $s \in S\setminus S_0$. For such a prime power $q_s$, Lemma \ref{lem-the-set-U-q-poonen-lemma} implies that $U_{q_s^d} + U_{q_s^d} = \bF_{q_s^d}$. Since $S_0$ is a finite set, $S\setminus S_0$ belongs in the ultrafilter $\cD$. Thus applying \L{}o\'s' theorem, we deduce from part (i) that
\begin{align*}
\cU_P + \cU_P = \prod_{s \in S}U_{q_s^d}/\cD + \prod_{s \in S}U_{q_s^d}/\cD = \prod_{s \in S}(U_{q_s^d} + U_{q_s^d})/\cD = \prod_{s\in S}\bF_{q_s^d}/\cD = \GF(\kappa^d),
\end{align*}
which proves the last assertion of part (i).

Part (ii) follows from the same arguments as in the proof of part (i).

\end{proof}

We will need the following three results.

\begin{proposition}
\label{prop-norm-of-elements-in-central-simple-algebra}
(see \cite[Proposition 2.6.3]{GS2017})

Let $k$ be a field, and let $A$ be a central simple $k$-algebra of degree $n$ (or equivalently of rank $n^2$). Let $K$ be a commutative $k$-subalgebra of $A$ that is a field extension of degree $n$ over $k$. Then for any $x \in K$,
\begin{align*}
\text{Norm}_A(x) = \text{Norm}_{K/k}(x)
\end{align*}
and
\begin{align*}
\text{Trace}_A(x) = \text{Trace}_{K/k}(x),
\end{align*}
where $\text{Norm}_A$, $\text{Trace}_A$ are the reduced norm and trace maps of $A$, respectively, and $\text{Norm}_{K/k}$, $\text{Trace}_{K/k}$ are the usual norm and trace maps of the field extension $K/k$.

\end{proposition}

\begin{lemma}
\label{lem-quadratic-fields-split-central-simple-algebra-serre}
(see Serre \cite[Corollary 3, p.194]{serre-local-fields})

Let $K$ be a field that is complete under a discrete valuation $v$ such that its residue field is quasi-finite, and let $\mathrm{Br}K$ denote the Brauer group of $K$. Let $A$ be a central simple $K$-algebra of degree $n$ (or equivalently of rank $n^2$), and let $[A] \in \mathrm{Br}K$ be the corresponding element of the Brauer group. Then $[A|$ has order $n$ in $\mathrm{Br}K$, and every extension of $K$ of degree $n$ can be embedded in $A$.

\end{lemma}

The following is an analogue of the Hasse--Minskowski local-to-global principle for quadratic forms over $\bK$.

\begin{theorem}
\label{thm-analogue-of-the-Hasse-Minskowksi-theorem}
(see Andriychuk \cite[Theorem 2, p.3]{Andriychuk2004})

A nondegenerate quadratic form $\cQ$ over $\bK$ is isotropic if and only if it is isotropic over all the $P$-adic completions $\bK_P$ of $\bK$.

\end{theorem}

\begin{lemma}
\label{lem-the-set-S-a-b}

Let $a, b$ be elements in $\bK^{\times}$, and let $P \in \bP\cup \{\infty\}$ be a prime in $\bK$. Then
\begin{itemize}

\item [(i)] if $P \in \Delta_{a, b}$, then $\red_P^{-1}(\cU_P) \subseteq \cS_{a, b}(\bK_P) \subseteq \cO_{\bK_P}$, where $\cO_{\bK_P}$ is the valuation ring of $\bK_P$  in \ref{itm: N2}.

\item [(ii)] if $P \not\in \Delta_{a, b}$, then $\cS_{a, b}(\bK_P) = \bK_P$.

\item [(ii)] $\cS_{a, b} = \bK \bigcap_{P \in \Delta_{a, b}}\cS_{a, b}(\bK_P)$.

\end{itemize}

\end{lemma}

\begin{proof}

Throughout the proof, $\text{Norm}$, $\text{Trace}$ denote the reduced norm and trace maps of $\bH_{a, b}\otimes \bK_P$, respectively.

We first prove the following claim:

$\star$ \textbf{Claim A.} \textit{Let $P \in \Delta_{a, b}$. For any element $\epsilon \in \bK_P$, the polynomial $f_{\epsilon}(x) = x^2 - \epsilon x + 1$ is a reduced characteristic polynomial of an element of $\bH_{a, b}\otimes \bK_P$ if and only if it is a power of a monic irreducible polynomial in $\bK_P[x]$.}

Indeed, write
\begin{align*}
\bH_{a, b} = \bK\cdot 1 + \bK \cdot \alpha + \bK \cdot \beta + \bK \cdot \alpha\beta,
\end{align*}
where $\alpha^2 = a$, $\beta^2 = b$, and $\alpha\beta = -\beta\alpha$.

Suppose that $f_{\epsilon}(x)$ is a reduced characteristic polynomial of an element of $\bH_{a, b}\otimes \bK_P$, that is, there exists an element $\fq = u + \alpha v + \beta w + \alpha\beta z \in \bH_{a, b}\otimes \bK_P$ for some elements $u, v, w, z \in \bK_P$ such that
\begin{align}
\label{e-eqn2-main-lem2}
\Norm(\fq) &= u^2 - a v^2 - bw^2 + abz^2 = 1,
\end{align}
and
\begin{align}
\label{e-eqn3-main-lem2}
\Trace(\fq) &= 2u = \epsilon.
\end{align}

 If $\epsilon^2 - 4$ is a nonsquare element in $\bK_P$, then $f_{\epsilon}(x)$ is irreducible over $\bK_P$.

So without loss of generality, assume that $\epsilon^2 - 4$ is a square in $\bK_P$, and thus $\epsilon^2 - 4 = \lambda^2$ for some element $\lambda \in \bK_P$. By (\ref{e-eqn2-main-lem2}), (\ref{e-eqn3-main-lem2}), we deduce that
\begin{align*}
\dfrac{\epsilon^2}{4} - 1 - a v^2 = bw^2 - abz^2 = b(w^2 - az^2).
\end{align*}
Since $\epsilon^2/4 - 1 = \lambda^2/4$, we deduce that
\begin{align}
\label{e-eqn4-main-lem2}
\left(\dfrac{\lambda}{2}\right)^2 - av^2 = \Norm\left(\dfrac{\lambda}{2} - \sqrt{a}v\right) = b\Norm(w - \sqrt{a}z).
\end{align}

Since $\bH_{a, b}$ is nonsplit (or ramified) at $P$, $a$ is a nonsquare element in $\bK_P$.

If $w^2 - az^2 = 0$, then (\ref{e-eqn4-main-lem2}) implies that $\left(\dfrac{\lambda}{2}\right)^2 - av^2 = 0$. Since $a$ is a nonsquare element in $\bK_P$, we deduce that $\lambda = v = 0$, which implies that $\epsilon = \pm 2$. Thus
\begin{align*}
f_{\epsilon}(x) = x^2 \pm 2 x + 1  = (x \pm 1)^2,
\end{align*}
which proves our claim.

If $w^2 - az^2 \ne 0$, we see from (\ref{e-eqn4-main-lem2}) that
\begin{align*}
b = \Norm\left(\dfrac{\lambda/2 - \sqrt{a}v}{w - \sqrt{a}z}\right),
\end{align*}
which implies that $\bH_{a, b}\otimes \bK_P$ is split (or unramified), a contradiction.

We now prove the converse. Suppose that $f_{\epsilon}(x) = x^2 - \epsilon x + 1$ is a power of an irreducible polynomial in $\bK_P$.

If $f_{\epsilon}(x) = (u x + v)^2$ for some elements $u, v\in \bK_P$, then $u = 1$ and $v = 1$, and thus $\epsilon = -2$. In this case, we see that $f_{\epsilon}(x) = x^2 + 2x + 1$ is a reduced characteristic polynomial of $-1$.

Suppose that $f_{\epsilon}(x)$ is irreducible in $\bK_P[x]$. So $\bK_P[x]/(f_{\epsilon}(x))$ is a quadratic field extension of $\bK_P$, and it thus follows from Lemma \ref{lem-quadratic-fields-split-central-simple-algebra-serre} that there exists an embedding of $K_P[x]/(f_{\epsilon}(x))$ into $\bH_{a, b}\otimes \bK_P$. Write $\mathfrak{r} \in \bH_{a, b}\otimes \bK_P$ for the image of $x$ under the composition of maps $\bK_P[x] \to \bK_P[x]/(f_{\epsilon}(x)) \hookrightarrow \bH_{a, b}\otimes \bK_P$. By Lemma \ref{prop-norm-of-elements-in-central-simple-algebra}, $\text{Norm}(\mathfrak{r}) = 1$ and $\text{Trace}(\mathfrak{r}) = \epsilon$, which proves our claim.

We are now ready to prove part (i). Suppose that $P \in \Delta_{a, b}$. Take an arbitrary $\lambda \in \red_P^{-1}(\cU_P) \subseteq \cO_{\bK_P}$, and thus $\epsilon = \red_P(\lambda) \in \cU_P \subseteq \GF(\kappa)_P$. Since the polynomial $x^2 -\epsilon x + 1$ is irreducible over $\GF(\kappa)_P$, it follows from Gauss' lemma (see Lang \cite{Lang}) that $x^2 - \lambda x + 1$ is irreducible over $\bK_P$. Thus using Claim A above, we deduce that $x^2 - \lambda x + 1$ is a reduced characteristic polynomial of some element of $\bH_{a, b}\otimes \bK_P$, and thus $\lambda \in \cS_{a, b}(\bK_P)$. Thus $\red_P^{-1}(\cU_P) \subseteq \cS_{a, b}(\bK_P)$.

We prove the second inclusion of part (i). Let $\cM_{\bK_P}$ denote the maximal ideal of $\cO_{\bK_P}$ that consists of all elements $\alpha$ of $\cO_{\bK_P}$ whose valuations $v_P(\alpha)$ are positive. If $\alpha \in \bK_P\setminus \cO_{\bK_P}$, then we contend that $f_{\alpha}(x) = x^2 - \alpha x + 1$ has two distinct roots, and thus is a product of two distinct linear factors. Indeed, since $\alpha \in \bK_P\setminus \cO_{\bK_P}$, $\alpha^{-1} \in \cO_{\bK_P}$, and thus $\alpha = \dfrac{1}{\beta}$ for some element $\beta \in \cO_{\bK_P}$ such that $v_P(\beta) > 0$.

We see that
\begin{align*}
f_{\alpha}(x) = \dfrac{g_{\beta}(x)}{\beta},
\end{align*}
where
\begin{align*}
g_{\beta}(x) = \beta x^2 -  x + \beta.
\end{align*}

Since $v_P(\beta) > 0$, $g_{\beta}(0) = \beta \equiv 0 \pmod{\cM_{\bK_P}}$. On the other hand, $g'(0) = -1 \not\equiv 0 \pmod{\cM_{\bK_P}}$. Thus by Hensel's lemma (see \cite[Theorem 1.3.1]{prestel}), $g_{\beta}(x)$ has a root $x_0 \in \cO_{\bK_P}$ such that $x_0 \equiv 0 \pmod{\cM_{\bK_P}}$. Thus $x_0$ is a root of $f_{\alpha}(x)$. The other root $x_1 \in \bK_P$ of $f_{\alpha}(x)$ satisfies $x_0x_1 = 1$, and thus $x_1 \ne x_0$ since $v_P(x_1) = -v_P(x_0) < 0 < v_P(x_0)$. Therefore $x^2 - \alpha x + 1$ is a product of two distinct linear polynomials for any element $\alpha \in \bK_P\setminus \cO_{\bK_P}$. By \textbf{Claim A} above, $\alpha$ does not belong in $\cS_{a, b}(\bK_P)$. Thus $\cS_{a, b}(\bK_P) \subseteq \cO_{\bK_P}$, which proves part (i).

Part (ii) follows a similar proof of Park \cite[Lemma 2.2(a)]{park}.

Using part (ii) and since $\bK \subseteq \bK_P$ for any prime $P$, we see that
\begin{align*}
\bK \bigcap \left(\bigcap_{P \in \Delta_{a, b}} \cS_{a, b}(\bK_P)\right) = \bK \bigcap \left(\bigcap_{P \in \bP \cup \{\infty\}} \cS_{a, b}(\bK_P)\right),
\end{align*}
and thus part (iii) follows from an analogue of the Hasse--Minskowski local-to-global principle (see Theorem \ref{thm-analogue-of-the-Hasse-Minskowksi-theorem}).

\end{proof}

\begin{lemma}
\label{lem-the-set-T-a-b}

For any $a, b \in \bK^{\times}$,
\begin{align*}
\cT_{a, b} = \bigcap_{P \in \Delta_{a, b}}\cO_{\bK_P} = \bigcap_{P \in \Delta_{a, b}}\bA_P.
\end{align*}
In particular, $\cT_{a, b}$ is a commutative ring.

\end{lemma}

\begin{proof}

Since $\cO_{\bK_P} \cap \bK = \bA_P$ and $\cT_{a, b} \subseteq \bK$, $\bigcap_{P \in \Delta_{a, b}}\cO_{\bK_P} = \bigcap_{P \in \Delta_{a, b}}\bA_P$ follows immediately. So it suffices to prove that
\begin{align*}
\cT_{a, b} = \bigcap_{P \in \Delta_{a, b}}\cO_{\bK_P}.
\end{align*}

By parts (i) and (iii) in Lemma \ref{lem-the-set-S-a-b}, we see immediately  that
\begin{align*}
\cS_{a, b} \subseteq \bigcap_{P \in \Delta_{a, b}}\cS_{a, b}(\bK_P) \subseteq \bigcap_{P \in \Delta_{a, b}}\cO_{\bK_P},
\end{align*}
and thus
\begin{align}
\label{e-eqn1-the-set-T-a-b-lemma}
\cT_{a, b} = \cS_{a, b} + \cS_{a, b} \subseteq \bigcap_{P \in \Delta_{a, b}}\cO_{\bK_P}.
\end{align}

For the reverse inclusion, set $\cV_P = \red_P^{-1}(\cU_P)$ for each prime $P \in \bP \cup \{\infty\}$. By Lemma \ref{lem-cU-P-is-the-ultraproduct-of-U-q}, we obtain that
\begin{align}
\label{e-eqn2-the-set-T-a-b-lemma}
\cV_P + \cV_P = \red_P^{-1}(\GF(\kappa)_P) = \cO_{\bK_P}.
\end{align}

Take an arbitrary element $\alpha \in \bigcap_{P \in \Delta_{a, b}}\cO_{\bK_P}$. For each prime $P \in \Delta_{a, b}$, we know from (\ref{e-eqn2-the-set-T-a-b-lemma}) that there exists an element $\beta_P \in \cO_{\bK_P}$ such that both $\beta_P$ and $\alpha - \beta_P$ belong in $\cV_P$. Since $\Delta_{a, b}$ is a finite set (see Proposition \ref{pro-description-of-Delta-a-b}), we deduce from the Strong Approximation Theorem for $\bK$ \footnote{The strong approximation theorem given in Endo \cite[Theorem 3, p.84]{endo} applies to every algebraic function field over an arbitrary constant field that is not necessarily finite, and so the rational function field $\bK$ considered in our paper is a special case.}(see Endo \cite[Theorem 3, p.84]{endo})) that there exists an element $\beta \in \bK$ such that both $\beta$ and $\alpha - \beta$ belong in $\cV_P = \red_P^{-1}(\cU_P)$ for all $P \in \Delta_{a, b}$. Thus we deduce from Lemma \ref{lem-the-set-S-a-b}(i) that both $\beta$ and $\alpha - \beta$ belong $\cS_{a, b}$, and thus
\begin{align*}
\alpha = \beta + (\alpha - \beta) \in \cS_{a, b} + \cS_{a, b} = \cT_{a, b}.
\end{align*}
Therefore $\bigcap_{P \in \Delta_{a, b}}\cO_{\bK_P} \subseteq \cT_{a, b}$. The lemma thus follows from (\ref{e-eqn1-the-set-T-a-b-lemma}).

\end{proof}

\begin{remark}
\label{rem-Daans-phd-thesis}

In his Ph.D. thesis, Daans proves an analogue of the above result. Let $\bF$ be a global field or a function field in one variable over a pseudo-algebraically closed field $F$ such that $F$ is either perfect or of odd characteristic. Let $q(x, y, z, w)$ be a totally indefinite $2$-fold Pfister form defined over $\bF$ of the form $x^2 - ay^2 - bz^2 + abw^2$ for some elements $a, b \in \bF$. Note that $q$ is the norm form of the quaternion algebra $\bH_{a, b}$. Let $\cS(q)$ be the set of elements $d \in \bF$ such that $q$ is isotropic over the splitting field of $t^2 - dt + 1$. Note that $\cS(q)$ is an analogue of $\cS_{a, b}$. Let $\Delta_q$ denote the set of $\bZ$-valuations on $\bF$ for which $q$ is anisotropic over the henselisation $\bF_v$. Then Daans' Theorem (see \cite[Theorem 6.3.10, p.128]{daans-phd-thesis}) implies that
\begin{align*}
S(q) + S(q) = \bigcap_{v \in \Delta_q}\cO_v,
\end{align*}
where $\cO_v$ is the valuation ring of $\bF_v$.

\end{remark}

\begin{lemma}
\label{lem-T-is-a-b-Diophantine}

Let $a, b$ be elements in $\bK^{\times}$. Then
\begin{itemize}

\item [(i)] $\cT_{a, b}$ is Diophantine and defined by $7$ quantifiers.

\item [(ii)] $\cT_{a, b}^{\times}$ is Diophantine and defined by $7$ quantifiers.

\end{itemize}

\end{lemma}

\begin{proof}

Since $\cS_{a, b}$ is Diophantine and defined by $3$ quantifiers, part (i) follows immediately from the definition of $\cT_{a, b}$.

Take an arbitrary prime $P \in \bP \cup \{\infty\}$. For any $\alpha \in \bK^{\times}$, we see that $\alpha \in \cO_{\bK_P}^{\times}$ if and only if $\dfrac{\alpha^2 + 1}{\alpha} \in \cO_{\bK_P}$. It thus follows from Lemma \ref{lem-the-set-T-a-b} that $\alpha \in \cT_{a, b}^{\times}$ if and only if $\dfrac{\alpha^2 + 1}{\alpha} \in \cT_{a, b}$. Therefore part (ii) follows immediately.

\end{proof}

The following result is an analogue of Lemma 3.14 given in Park \cite{park} for the rational function field $\bK$.

\begin{lemma}
\label{lem-Lemma-3.14-in-Park}

For any $a, b \in \bK^{\times}$,
\begin{align*}
\bK^{\times 2} \cdot \cT_{a, b}^{\times} = \bigcap_{P \in \Delta_{a, b}}v_P^{-1}(2\bZ).
\end{align*}

\end{lemma}

\begin{proof}

The proof of Lemma 3.14 given in Park \cite{park} essentially relies on a similar result of Lemma \ref{lem-the-set-T-a-b} for a number field $K$ that is Proposition 2.3 in Park \cite{park}, and the weak approximation theorem for global fields. We can replace Proposition 2.3 in Park \cite{park} by Lemma \ref{lem-the-set-T-a-b} and the weak approximation theorem for number fields by the weak approximation theorem for valued fields (see Engler--Prestel \cite[Theorem 2.4.1]{prestel}) or by the strong approximation theorem in Endo \cite[Theorem 3]{endo}, and following essentially the proof of Park \cite[Lemma 3.14]{park}, the lemma follows immediately.

\end{proof}

For each element $c \in \bK^{\times}$, define
\begin{align}
\label{e-def-I-a-b-^c-definition}
\cI_{a, b}^c = c \cdot \bK^2 \cdot \cT_{a, b}^{\times} \bigcap (1 - \bK^2 \cdot \cT_{a, b}^{\times}).
\end{align}

Repeating the proof of Lemma 3.15 in Park \cite{park}, and using Lemma \ref{lem-Lemma-3.14-in-Park}, the following follows immediately.

\begin{lemma}
\label{lem-Lemma-3.15-in-Park}

For any $a, b, c \in \bK^{\times}$,
\begin{align*}
\cI_{a, b}^c &= \{\alpha \in \bK \; | \; \text{$v_P(\alpha)$ is an odd positive integer for all $P \in \Delta_{a, b} \cap \bP(c)$,}\\
 &\text{and both $v_P(\alpha)$ and $v_P(1 - \alpha)$ are even for all $P \in \Delta_{a, b} \setminus \bP(c)$}\}.
\end{align*}

\end{lemma}

For any elements $a, b \in \bK^{\times}$, define
\begin{align}
\label{e-def-Jacobson-radical-in-1st-definition}
\cJ_{a, b} = \bigcap_{P \in \Delta_{a, b}}P\cO_{\bK_P}.
\end{align}

Following along the lines of the proof of Lemma 3.17 in Park \cite{park} in which we replace the weak approximation theorem for global fields by the one for valued fields (see Engler--Prestel \cite[Theorem 2.4.1]{prestel}) or by the strong approximation theorem in Endo \cite[Theorem 3]{endo}, the following follows from Lemma \ref{lem-Lemma-3.14-in-Park}.

\begin{lemma}
\label{lem-Diophantine-def-of-J-a-b}

For any elements $a, b \in \bK^{\times}$,
\begin{align*}
\cJ_{a, b} &= \{0\} \cup \{\alpha \in \bK^{\times}\; |\; \text{$\exists \beta_1, \beta_2 \in \bK$ such that} \\
&\text{$\beta_1, \alpha - \beta_2 \in a \cdot \bK^2 \cdot \cT_{a, b}^{\times} \bigcap (1 - \bK^2 \cdot \cT_{a, b}^{\times})$, and}\\
&\text{$\beta_1, \alpha - \beta_2 \in b \cdot \bK^2 \cdot \cT_{a, b}^{\times} \bigcap (1 - \bK^2 \cdot \cT_{a, b}^{\times})$}\}.
\end{align*}

\end{lemma}

\begin{remark}

We can also use Daans \cite[Lemma 5.4]{daans-21} to prove the above lemma.

\end{remark}

Note that the set $\cJ_{a, b}$ is the Jacobson radical $\cJ\left(\cT_{a, b}\right)$ of the ring $\cT_{a, b}$ in \ref{itm: N8} (see Lemma \ref{lem-the-set-T-a-b}). Indeed, by Lemma \ref{lem-the-set-T-a-b}, we see that
\begin{align*}
 \cJ\left(\cT_{a, b}\right) = \bigcap_{P \in \Delta_{a, b}}P\bA_P.
 \end{align*}

 We know that $P\cO_{\bK_P} \cap \bK = P\bA_P$ for all primes $P \in \Delta_{a, b}$, and since $\cJ_{a, b}$ is a subset of $\bK$ (see Lemma \ref{lem-Diophantine-def-of-J-a-b}), we deduce that
 \begin{align*}
 \cJ_{a, b} = \cJ_{a, b}\cap \bK = \bigcap_{P \in \Delta_{a, b}}(P\cO_{\bK_P}\cap \bK) = \bigcap_{P \in \Delta_{a, b}}P\bA_P = \cJ\left(\cT_{a, b}\right),
 \end{align*}
 which proves our claim. Thus we obtain the following.

 \begin{corollary}
 \label{cor-the-Jacobson-radical-is-Diophantine}

 The Jacobson radical $\cJ\left(\cT_{a, b}\right)$ is Diophantine, and given by
 \begin{align*}
\cJ\left(\cT_{a, b}\right) &= \{0\} \cup \{\alpha \in \bK^{\times}\; |\; \text{$\exists \beta_1, \beta_2 \in \bK$ such that} \\
&\text{$\beta_1, \alpha - \beta_2 \in a \cdot \bK^2 \cdot \cT_{a, b}^{\times} \bigcap (1 - \bK^2 \cdot \cT_{a, b}^{\times})$, and}\\
&\text{$\beta_1, \alpha - \beta_2 \in b \cdot \bK^2 \cdot \cT_{a, b}^{\times} \bigcap (1 - \bK^2 \cdot \cT_{a, b}^{\times})$}\}.
\end{align*}
In particular, $\cJ\left(\cT_{a, b}\right)$ can be defined using at most $66$ existential quantifiers.

 \end{corollary}

\begin{proof}

By Lemma \ref{lem-T-is-a-b-Diophantine}, $\cT_{a, b}^{\times}$ is Diophantine, and defined by $7$ quantifiers. Thus it follows from Lemma \ref{lem-Diophantine-def-of-J-a-b} that $\cJ\left(\cT_{a, b}\right)$ is Diophantine, and requires at most $66$ existential quantifiers.

\end{proof}

Using essentially the same arguments as in the proof of Eisentr\"ager--Morrison \cite[Lemma 3.19]{EM18} and the above corollary, we deduce the following.

\begin{corollary}
\label{cor-R-tilde-a-b-is-Diophantine}

For any $a, b \in \bK^{\times}$,
\begin{align*}
\widetilde{\cR_{a, b}} = \{\alpha \in \bK \; | \; \text{$\alpha = 0$ or $\alpha^{-1} \in \bK\setminus \cJ(\cT_{a, b})$}\}.
\end{align*}
In particular, $\widetilde{\cR_{a, b}}$ is universally definable in $\bK$, and requires at most $67$ universal quantifiers.

\end{corollary}

\subsection{Integrity at a prime for $\bK = \GF(\kappa)(t)$}
\label{subsec-integrity-at-a-prime-for-K}

In \cite{Andriychuk1999}, a field $k$ is called \textbf{pseudofinite} if $k$ satisfies the following:
\begin{itemize}

\item [(i)] $k$ is perfect;

\item [(ii)] $k$ has a unique extension of each degree; and

\item [(iii)] $k$ is pseudo-algebraically closed, i.e., every absolutely irreducible variety over $k$ has a $k$-rational point (see \cite{FJ}).

\end{itemize}

We recall the following result due to Andriychuk \cite{Andriychuk1999}.

\begin{theorem}
\label{thm-andriychuk-Hasse-exact-sequence}
(see Andriychuk \cite[Proposition 1]{Andriychuk1999}))

Let $K$ be an algebraic function field in one variable over a pseudofinite field $k$. Then there is an exact sequence
\begin{align*}
0 \to \mathrm{Br}K \to \oplus_{v \in V_K}\mathrm{Br}K_v \to \bQ/\bZ \to 0,
\end{align*}
where $V_K$ denotes the set of nonequivalent nontrivial absolute values of $K$, and $\mathrm{Br}K$, $\mathrm{Br}K_v$ denote the Brauer groups of $K$, $K_v$, respectively.

\end{theorem}

\begin{remark}
\label{rem-GF(kappa)-is-pseudo-finite}

Note that $\GF(\kappa)$ is pseudofinite, and thus $\bK = \GF(\kappa)(t)$ satisfies the conditions in Theorem \ref{thm-andriychuk-Hasse-exact-sequence}.

\end{remark}

\begin{theorem}
\label{thm-BrK=Q/Z-Serre}
(See Serre \cite[Proposition 6, p.193 and Corollary 1, p.194]{serre-local-fields}

Let $K$ be a field that is complete under a discrete valuation $v$ such that its residue field is quasi-finite, and let $\mathrm{Br}K$ denote the Brauer group of $K$. Then
\begin{itemize}

\item [(i)] $\mathrm{Br}K$ is isomorphic to $\bQ/\bZ$.

\item [(ii)] Let $L$ be a finite extension of $K$ of degree $n$. A necessary and sufficient condition for an element $a \in \mathrm{Br}K$ to be split by $L$ is $na = 0$.

\end{itemize}

\end{theorem}

The following is an analogue of Eisentr\"ager \cite[Theorem 3.1]{Eis} for integrality at a prime for $\bK$ that we will need in the proof of our main theorem.

\begin{theorem}
\label{thm-integrality-at-a-prime-for-K}

Let $P \in \bP \cup \{\infty\}$ be a prime in $\bK = \GF(\kappa)(t)$. Then the set $\{\alpha \in \bK \; | \; v_P(\alpha) \ge 0\}$ is Diophantine over $\bK$, and defined using $9$ existential quantifiers.

\end{theorem}

\begin{proof}

In \cite{Eis},  Eisentr\"ager proves that if $K$ is a global field and $\fp$ be a nonarchimedean prime of $K$, then the set $\{x \in K \; | \; v_{\fp}(x) \ge 0\}$ is Diophantine over $K$. Her proof is essentially based on the following three results:
\begin{itemize}

\item [(i)] there exists a $4$-dimensional central division algebra over $K$ that is ramified exactly at two given distinct nonarchimedean primes. This result is based on the weak approximation theorem for global fields, the Hasse exact sequence for global fields that is similar to that in Theorem \ref{thm-andriychuk-Hasse-exact-sequence} with $K$ replaced by global fields, and the fact that the Brauer group of a global field is isomorphic to $\bQ/\bZ$.

\item [(ii)] any quadratic extension field of a local field $K_{\fp}$ is a splitting field for a $4$-dimensional central division algebra over $K_{\fp}$.

\item [(iii)] let $F$ be an arbitrary field, and $L$ be a quadratic extension of $F$. Then $L$ is a splitting field for a quaternion algebra $A$ if and only if there exists an injective $F$-algebra homomorphism $L \hookrightarrow A$ (see Voight \cite[Lemma 5.4.7]{voight})

\end{itemize}

The third result above applies to any field, and so it also applies to $\bK$. For the first two results above, we can obtain similar results for $\bK$ by applying Theorem \ref{thm-andriychuk-Hasse-exact-sequence} in place of the classical Hasse exact sequence, and using Theorem \ref{thm-BrK=Q/Z-Serre}, and Lemma \ref{lem-quadratic-fields-split-central-simple-algebra-serre}. Thus following the same lines of the proof of \cite[Theorem 3.1]{Eis}, the theorem follows immediately.

\end{proof}

\subsection{A first-order definition of $\bA$ in $\bK$}
\label{subsec-defining-A-in-K}

In this subsection, we characterize polynomials in $\bA$ in $\bK$ with a universal-existential formula. We first describe nonsquares in the $\infty$-adic completion $\bK_{\infty} = \GF(\kappa)((1/t))$.

\begin{lemma}
\label{lem-nonsquares-of-K-infinity}

Any nonsquare element in $\bK_{\infty} = \GF(\kappa)((1/t))$ has one of the following forms: $\dfrac{\alpha^2}{t}$, $h\alpha^2$, or $\dfrac{h\alpha^2}{t}$ for some element $\alpha \in \bK_{\infty}$ and $h \in \GF(\kappa)$.

\end{lemma}

\begin{proof}

Note that the residue field of $\bK_{\infty}$ is $\GF(\kappa)$. Thus the lemma immediately follows from Hensel's lemma for $\bK_{\infty}$ (see \cite{prestel}). See also Daans \cite[Proposition 1.4.5]{daans-master-thesis} or Tyrrell \cite[Lemma B.0.1]{tyrell-thesis}.

\end{proof}

\begin{lemma}
\label{lem-ramified-quaternion-algebra-over-K-infinity}

Let $g, h$ be elements in $\GF(\kappa)^{\times}$ such that $g$ is a nonsquare element in $\GF(\kappa)$. Let $a = g$ and $b = \dfrac{h}{t}$ be elements in $\bK$. Then the quaternion algebra
\begin{align*}
\bH_{a, b}\otimes \bK_{\infty} = \bK_{\infty} \cdot 1 \oplus \bK_{\infty} \cdot \alpha \oplus \bK_{\infty} \cdot \beta \oplus \bK_{\infty} \cdot \alpha\beta
\end{align*}
is nonsplit (or ramified), where $\alpha^2 = a$, $\beta^2 = b$, and $\alpha\beta = -\beta\alpha$.

\end{lemma}

\begin{proof}

By Proposition \ref{prop-basic-properties-local-symbols}(iii), we know that $\bH_{a, b}\otimes \bK_{\infty}$ is ramified if and only if the local symbol in Corollary \ref{cor-Nguyen-formula-for-local-symbols-K_P-quadratic-symbol} $(a, b)_{\infty}$ is equal to $-1$. By the same corollary, we know that
\begin{align*}
(a, b)_{\infty} = \left((-1)^{v_{\infty}(a)v_{\infty}(b)} \red_{\infty}\left(\dfrac{a^{v_{\infty}(b)}}{b^{v_{\infty}(a)}} \right)\right)^{\frac{\kappa - 1}{2}}.
\end{align*}

We know that $v_{\infty}(a) = 0$ and $v_{\infty}(b) = v_{\infty}(h/t) = 1$. It thus follows from the above equation that
\begin{align*}
(a, b)_{\infty} = \red_{\infty}(a)^{\frac{\kappa - 1}{2}} = a^{\frac{\kappa - 1}{2}},
\end{align*}
since $a = g \in \GF(\kappa)$. Since $a = \ulim_{s \in S}a_s$ is a nonsquare in $\GF(\kappa) = \prod_{s\in S}\bF_{q_s}/\cD$, where the $a_s$ belong in $\bF_{q_s}$, \L{}o\'s' theorem implies that $a_s$ is a nonsquare in $\bF_{q_s}$ for $\cD$-almost all $s \in S$. Thus $a_s^{\frac{q_s - 1}{2}} = -1$ for $\cD$-almost all $s \in S$, and since $\kappa = \ulim_{s \in S}q_s$, we deduce that
\begin{align*}
a^{\frac{\kappa - 1}{2}} = \ulim_{s\in S}a_s^{\frac{q_s - 1}{2}} = -1.
\end{align*}
Therefore
\begin{align*}
(a, b)_{\infty} = -1,
\end{align*}
which proves the lemma.

\end{proof}

\begin{remark}
\label{rem-nonsquares-in-GF(kappa)}

In the above proof, we have showed that  $a$ is a nonsquare in $\GF(\kappa)$ if and only if $a^{\frac{\kappa - 1}{2}} = -1$.

\end{remark}

We recall the following result of Nguyen (see \cite[Lemma 5.6]{nguyen-ultrafinite-fields-2023}), specialized for the quadratic residue symbol in $\bA$.
\begin{lemma}
\label{lem-power-residue-symbol-for-constants}

Let $a \in \GF(\kappa)$, and let $P$ be a monic irreducible polynomial of positive degree in $\bA$. Then the quadratic residue symbol $\left( \frac{a}{P} \right)$ satisfies the following.
\begin{itemize}

\item [(i)] $\left( \frac{a}{P} \right) = a^{\frac{\kappa-1}{2}\deg(P)}$.

\item [(ii)] Assume further that $a$ is a nonsquare in $\GF(\kappa)$. If $\deg(P)$ is odd, then $\left( \frac{a}{P} \right) = -1$, and if $\deg(P)$ is even, then $\left( \frac{a}{P} \right) = 1$.

\end{itemize}

\end{lemma}

\begin{proof}

Part (i) is exactly \cite[Lemma 5.6]{nguyen-ultrafinite-fields-2023}. The second part follows immediately from part (i) and Remark \ref{rem-nonsquares-in-GF(kappa)}.

\end{proof}

Let $p$ be an arbitrary odd prime, and let $q$ be a power of $p$. Let $\cP(q)_{\le k}$ be the space of polynomials of degree at most $k$ over $\bF_q$, and let $\cM(k, q) \subseteq \cP(q)_{\le k}$ the subset of monic polynomials of degree $k$. For a polynomial $f \in \bF_q[t]$ with $\deg(f) = k$, we set $\lVert f \rVert_q = q^k$.

Let
\begin{align*}
\pi_q(k) := \#\{P \in \cM(k, q) \; | \; \text{$P$ is irreducible}\}
\end{align*}
By the prime polynomial theorem (see Rosen \cite[Theorem 5.12]{rosen}),
\begin{align}
\label{eq-prime-polynomial-theorem}
\pi_q(k) = \dfrac{q^k}{k} + O\left(\dfrac{q^{k/2}}{k}\right).
\end{align}

For relatively prime polynomials $c, f \in \bF_q[t]$, let
\begin{align}
\label{eq-primes-in-arithmetic-progressions-of-degree-k}
\pi_q(k; f, c) := \#\{Q = c + fg \in \cM(k, q) \; |\; \text{$Q$ is irreducible}\}.
\end{align}

We denote by $\Phi_q$ the Euler $\phi$-function for $\bF_q[t]$, i.e., $\Phi_q(f)$ is the number of units in $\bF_q[t]/f\bF_q[t]$ (see Rosen \cite[Proposition 1.7]{rosen}). Whenever $f$ is a prime in $\bF_q[t]$, $\bF_q[t]/f\bF_q[t]$ is a field, and thus
\begin{align*}
\Phi_q(f) = \lVert f \rVert - 1 = q^{\deg(f)} - 1.
\end{align*}

For an arbitrary nonzero polynomial $f \in \bF_q[t]$, we have
\begin{align}
\label{e-Euler-q-function}
\Phi_q(f) = \lVert f \rVert\prod_{P | f}\left(1 - \dfrac{1}{\lVert P \rVert}\right).
\end{align}

The following result is due to Bank and Bary-Soroker \cite[Corollary 2.6]{BB-2015}.

\begin{theorem}
\label{thm-BB-2015}
(See Bank and Bary-Soroker \cite{BB-2015})

Let $k$ be a positive integer. Then
\begin{align*}
\pi_q(k; f, c) \sim \dfrac{\pi_q(k)}{\Phi_q(f)}, \; q \to \infty
\end{align*}
holds uniformly for all relatively prime polynomials $c, f \in \bF_q[t]$ satisfying $\lVert f \rVert \le q^{k - 4}$.

\end{theorem}

Let $c, f$ be relatively prime polynomials in $\bA = \GF(\kappa)[t]$. Since $\GF(\kappa) =\prod_{s \in S}\bF_{q_s}/\cD$, using \L{}o\'s' theorem, one can write $c = \ulim_{s \in S}c_s$ and $f = \ulim_{s \in S}f_s$, where the $c_s$ and $f_s$ are polynomials of degrees $\deg(c)$ and $\deg(f)$ in $\bF_{q_s}[t]$, respectively for $\cD$-almost all $s\in S$. Since $c, f$ are relatively prime in $\bA$, it follows from \L{}o\'s' theorem that $c_s, f_s$ are relatively prime in $\bF_{q_s}[t]$ for $\cD$-almost all $s\in S$. For a sufficiently large integer $k$, it follows from Theorem \ref{thm-BB-2015} that there exists a finite set $S_0 \subset S$ such that
\begin{align*}
\pi_{q_s}(k; f_s, c_s) > \dfrac{\pi_{q_s}(k)}{2\Phi_{q_s}(f_s)}
\end{align*}
for all $s \in S\setminus S_0$. Thus for any absolute constant $M > 0$, since the sequence $\{q_s\}_{s\in S}$ is unbounded, $q_s > M$ for all $s \in S\setminus S_1$ for some finite subset $S_1$ of $S$. We see from (\ref{eq-prime-polynomial-theorem}) and (\ref{e-Euler-q-function}) that there exists an integer $k_0 > 0$ such that for all $k \ge k_0$,
\begin{align*}
\dfrac{\pi_{q_s}(k)}{2\Phi_{q_s}(f_s)} > M
\end{align*}
for all $s \in S\setminus (S_0 \cup S_1)$, and thus
\begin{align}
\label{e-lower-bound-of-pi-q-s}
\pi_{q_s}(k; f_s, c_s) > M.
\end{align}
for any $k \ge k_0$ and $s \in S \setminus (S_0 \cup S_1)$. For an arbitrary positive integer $k \ge k_0$, by (\ref{e-lower-bound-of-pi-q-s}), we choose $a_s \in \bF_{q_s}[t]$ such that $P_s = c_s + a_sf_s$ is a monic prime of degree $k$ in $\bF_{q_s}[t]$ for all $s \in S\setminus (S_0 \cup S_1)$. Since $S_0 \cup S_1$ is a finite set, $S_0 \cup S_1$ does not belong in a nonprincipal ultrafilter $\cD$, and thus $S\setminus (S_0 \cup S_1) \in \cD$. Hence $P_s$ is a monic prime of degree $k$ in $\bF_{q_s}[t]$ for $\cD$-almost all $s\in S$. Since the $P_s$ is of degree $k$, the element $P = \ulim_{s \in S}P_s$ belongs in $\bA$, and is of degree $k$. By \L{}o\'s' theorem, $P$ is a monic prime of degree $k$ in $\bA$. Furthermore, since $P_s = c_s + a_sf_s$, we see that
\begin{align*}
P = \ulim_{s\in S}P_s = \ulim_{s\in S}c_s + \ulim_{s\in S}a_s\ulim_{s\in S}f_s = c + af,
\end{align*}
where $a= \ulim_{s\in S}a_s$. Since the $a_s$ are polynomials in $\bF_{q_s}[t]$ of degrees $\le k$, we deduce that $a$ is a polynomial of degree $\le k$ in $\bA$. Thus we obtain the following.

\begin{corollary}
\label{cor-BB-thm-for-ultra-finite-fields}
(Analogue of Dirichlet's theorem on primes in arithmetic progression for $\bA$)

For any relatively prime polynomials $c, f \in \bA = \GF(\kappa)[t]$, there exists a positive integer $k_0 >0$ such that for any integer $k \ge k_0$, there exists a monic prime of degree $k$ of the form $P = c + af$ for some $a \in \bA$.

\end{corollary}

The following is an analogue of Lemma 2.5 in Tyrrell \cite{tyrell}.

\begin{lemma}
\label{lem-quaternion-algebra-ramified-exactly-at-a-given-prime-P-and-infinity-main-lem1}

Let $P$ be a monic prime in $\bA$, and let $a$ be a nonsquare element in $\GF(\kappa)$. Then there exists a monic prime $Q \in \bA$ that satisfies the following.
\begin{itemize}

\item [(i)] $Q$ is of opposite parity in degree to $P$.

\item [(ii)] the only ramified (or nonsplit) primes of $\bH_{aP, aQ}$ are exactly $P$ and the infinite prime $\infty$, i.e, both local symbols $(aP, aQ)_P$, $(aP, aQ)_{\infty}$ equal $-1$ and $(aP, aQ)_{\fp} = 1$ for all primes $\fp \ne P, \infty$.

\end{itemize}

\end{lemma}

\begin{remark}
\label{rem-quadratic-symbols-and-squares-in-residue-fields}

By Proposition \ref{prop-properties-of-power-residue-symbol}(iii) with $n = 2$, for any polynomial $\alpha \in \bA$, the quadratic residue symbol $\left(\dfrac{\alpha}{P}\right)$ is equal $-1$ or $1$, according to whether $\red_P(\alpha)$ is a nonsquare or a square in $\GF(\kappa)_P$, respectively.

\end{remark}

\begin{proof}

We consider the following cases.

$\star$ \textbf{Case 1. $P$ is of odd degree.}

Take an arbitrary polynomial $\beta$ in $\bA$ that is relatively prime to $P$. By Corollary \ref{cor-BB-thm-for-ultra-finite-fields} implies that for an arbitrarily large even integer $2k$, there exists a monic prime $Q$ of degree $2k$ in $\bA$ such that $Q \equiv \beta^2 \pmod{P}$. Thus part (i) follows immediately.

Since $\red_P(Q) = (\red_P(\beta))^2$ is a square in $\GF(\kappa)_P$, we deduce that $\left(\dfrac{Q}{P}\right) = 1$. By Lemma \ref{lem-power-residue-symbol-for-constants}, $\left(\dfrac{a}{P}\right) = a^{\frac{\kappa - 1}{2}\deg(P)} = -1$ since $a^{\frac{\kappa - 1}{2}} = - 1$ (see Remark \ref{rem-nonsquares-in-GF(kappa)}). Thus
\begin{align*}
\left(\dfrac{aQ}{P}\right) = \left(\dfrac{Q}{P}\right) \left(\dfrac{a}{P}\right) = -1.
\end{align*}
By Corollary \ref{cor-Nguyen-formula-for-local-symbols-K_P-quadratic-symbol}, (\ref{e-quadratic-residue-symbol-equal-reduction-to-a-power}) in Remark \ref{rem-residue-symbol-equal-reduction-to-a-power}, and since $v_P(aQ) = 0$, $v_P(aP) = 1$, we deduce that
\begin{align*}
(aP, aQ)_P &= \left((-1)^{v_P(aP)v_P(aQ)} \red_P\left(\dfrac{(aP)^{v_P(aQ)}}{(aQ)^{v_P(aP)}}\right)\right)^{\frac{\kappa^{\deg(P)} - 1}{2}} \\
&= \red_P(aQ^{-1})^{\frac{\kappa^{\deg(P)} - 1}{2}} \\
&= \left(\dfrac{aQ}{P}\right)^{-1}\\
&= -1.
\end{align*}
Thus $\bH_{aP, aQ}$ is nonsplit (or ramified) at $P$.

We now prove that $\bH_{aP, aQ}$ is split (or unramified) at $Q$. Since both $P, Q$ are monic, the Reciprocity Law \ref{thm-General-Reciprocity-Law} implies that
\begin{align*}
\left(\dfrac{Q}{P}\right)\left(\dfrac{P}{Q}\right) = (-1)^{\frac{\kappa - 1}{2}\deg(P)\deg(Q)} = 1,
\end{align*}
and thus
\begin{align*}
\left(\dfrac{P}{Q}\right) = 1.
\end{align*}
By Lemma \ref{lem-power-residue-symbol-for-constants} and Remark \ref{rem-nonsquares-in-GF(kappa)},
\begin{align*}
\left(\dfrac{a}{Q}\right) = a^{\frac{\kappa - 1}{2}\deg(Q)} = (-1)^{2k} = 1.
\end{align*}
Therefore
\begin{align*}
\left(\dfrac{aP}{Q}\right) = \left(\dfrac{a}{Q}\right)\left(\dfrac{P}{Q}\right) = 1.
\end{align*}

Thus we deduce from Corollary \ref{cor-Nguyen-formula-for-local-symbols-K_P-quadratic-symbol}, (\ref{e-quadratic-residue-symbol-equal-reduction-to-a-power}) in Remark \ref{rem-residue-symbol-equal-reduction-to-a-power}, and the above equation that
\begin{align*}
(aP, aQ)_Q &= \left((-1)^{v_Q(aP)v_Q(aQ)} \red_Q\left(\dfrac{(aP)^{v_Q(aQ)}}{(aQ)^{v_Q(aP)}}\right)\right)^{\frac{\kappa^{\deg(Q)} - 1}{2}} \\
&= \red_Q(aP)^{\frac{\kappa^{\deg(Q)} - 1}{2}} \\
&= \left(\dfrac{aP}{Q}\right)\\
&= 1.
\end{align*}
Therefore $\bH_{aP, aQ}$ is split (or unramified) at $Q$.

For any non-infinite prime $\fp \ne P, Q, \infty$, we know that $v_{\fp}(P) = v_{\fp}(Q) = 0$, and it thus follows from Corollary \ref{cor-Nguyen-formula-for-local-symbols-K_P-quadratic-symbol} that $(aP, aQ)_{\fp} = 1$, which proves that $\bH_{aP, aQ}$ is split at $\fp$.

By an analogue of the Hilbert reciprocity law \ref{thm-analogue-of-the-Hilbert-reciprocity-law}, we deduce that
\begin{align*}
-1 = -\prod_{\fp \in \bP\cup \{\infty\}}(aP, aQ)_{\fp} = -(aP, aQ)_P(aP, aQ)_{\infty} = (aP, aQ)_{\infty},
\end{align*}
and thus $\bH_{aP, aQ}$ is nonsplit (or ramified) at $\infty$. Therefore the only ramified primes of $\bH_{aP, aQ}$ are exactly $P$ and $\infty$.

$\star$ \textbf{Case 2. $P$ is of even degree $d$.}

We know that the residue field $\bA/P\bA$ is an extension of degree $d$ over $\GF(\kappa)$. Since $\GF(\kappa)$ has a unique extension of degree $d$ that is of the form $\GF(\kappa^d) = \prod_{s\in S}\bF_{q_s^d}/\cD$ (see Ax \cite{ax-1968} or Nguyen \cite[Theorem 3.35]{nguyen-ultrafinite-fields-2023}), we can identify $\bA/P\bA$ with $\GF(\kappa^d)$.

Choose a nonsquare element $\alpha = \ulim_{s\in S}\alpha_s \in \GF(\kappa^d)$ for some elements $\alpha_s \in \bF_{q_s^d}$. Such a nonsquare element $\alpha$ can be constructed by choosing nonsquare elements $\alpha_s$ in $\bF_{q_s^d}$ for $\cD$-almost all $s\in S$, and so \L{}o\'s' theorem implies that $\alpha$ is a nonsquare element in $\GF(\kappa^d)$. Choose a polynomial $\beta \in \bA$ such that $\red_P(\beta) = \alpha$. By Corollary \ref{cor-BB-thm-for-ultra-finite-fields}, for an arbitrarily large odd integer $2k + 1$, there exists a monic prime $Q$ of degree $2k + 1$ such that $Q \equiv \beta \pmod{P}$. Thus part (i) follows trivially.

We know that $\red_P(Q) = \red_P(\beta) = \alpha$ is a nonsquare element in the residue field $\GF(\kappa)_P = \GF(\kappa^d)$. By Remark \ref{rem-quadratic-symbols-and-squares-in-residue-fields}, we deduce that
\begin{align*}
\left(\dfrac{Q}{P}\right) = -1.
\end{align*}

By Lemma \ref{lem-power-residue-symbol-for-constants}, $\left(\dfrac{a}{P}\right) = a^{\frac{\kappa - 1}{2}\deg(P)} = 1$ since $a^{\frac{\kappa - 1}{2}} = - 1$ and $\deg(P)$ is even (see Remark \ref{rem-nonsquares-in-GF(kappa)}). Thus
\begin{align*}
\left(\dfrac{aQ}{P}\right) = \left(\dfrac{Q}{P}\right) \left(\dfrac{a}{P}\right) = -1.
\end{align*}
Using the same argument as in Case 1, we deduce that $\bH_{aP, aQ}$ is nonsplit (or ramified) at $P$.

We now prove that $\bH_{aP, aQ}$ is split (or unramified) at $Q$. Since both $P, Q$ are monic, the Reciprocity Law \ref{thm-General-Reciprocity-Law} implies that
\begin{align*}
\left(\dfrac{Q}{P}\right)\left(\dfrac{P}{Q}\right) = (-1)^{\frac{\kappa - 1}{2}\deg(P)\deg(Q)} = 1,
\end{align*}
and thus
\begin{align*}
\left(\dfrac{P}{Q}\right) = -1.
\end{align*}
By Lemma \ref{lem-power-residue-symbol-for-constants} and Remark \ref{rem-nonsquares-in-GF(kappa)},
\begin{align*}
\left(\dfrac{a}{Q}\right) = a^{\frac{\kappa - 1}{2}\deg(Q)} = (-1)^{2k + 1} = -1.
\end{align*}
Therefore
\begin{align*}
\left(\dfrac{aP}{Q}\right) = \left(\dfrac{a}{Q}\right)\left(\dfrac{P}{Q}\right) = 1.
\end{align*}

Thus we deduce from Corollary \ref{cor-Nguyen-formula-for-local-symbols-K_P-quadratic-symbol}, (\ref{e-quadratic-residue-symbol-equal-reduction-to-a-power}) in Remark \ref{rem-residue-symbol-equal-reduction-to-a-power}, and the above equation that
\begin{align*}
(aP, aQ)_Q &= \left((-1)^{v_Q(aP)v_Q(aQ)} \red_Q\left(\dfrac{(aP)^{v_Q(aQ)}}{(aQ)^{v_Q(aP)}}\right)\right)^{\frac{\kappa^{\deg(Q)} - 1}{2}} \\
&= \red_Q(aP)^{\frac{\kappa^{\deg(Q)} - 1}{2}} \\
&= \left(\dfrac{aP}{Q}\right)\\
&= 1.
\end{align*}
Thus $\bH_{aP, aQ}$ is split (or unramified) at $Q$.

For any non-infinite prime $\fp$, repeating in the same manner as in Case 1, we can deduce that $\bH_{aP, aQ}$ is split (or unramified) at $\fp$. Using an analogue of the Hilbert reciprocity law \ref{thm-analogue-of-the-Hilbert-reciprocity-law}, we can verify that $\bH_{aP, aQ}$ is nonsplit at the infinite prime $\infty$, and thus the lemma follows immediately.

\end{proof}

We obtain a generalization of Shlapentokh's theorem (see \cite[Theorem 4.4]{SHL94}) for existential definability of valuation rings in a global field to that of valuation rings in $\bK$.

\begin{corollary}
\label{cor-generalization-of-Poonen-Remark-2.6}

For any non-infinite prime $P \in \bP$ of $\bK$, the ring $\bA_P = \cO_{\bK_P} \cap \bK$ is positively existentially definable with parameters in $\bK$.

\end{corollary}

\begin{proof}

This proof essentially follows the proof of Poonen \cite[Remark 2.5]{poonen-2009}. For any non-infinite prime $P \in \bP$, using the similar arguments as in Lemma \ref{lem-quaternion-algebra-ramified-exactly-at-a-given-prime-P-and-infinity-main-lem1}, we can find elements $a, b, c, d \in \bK^{\times}$ and distinct non-infinite primes $Q_1, Q_2 \ne P$ such that $\Delta_{a, b} = \{P, Q_1\}$ and $\Delta_{c, d} = \{P, Q_2\}$. By Lemma \ref{lem-the-set-T-a-b}, $\cT_{a, b} = \bA_P \cap \bA_{Q_1}$ and $\cT_{c, d} = \bA_P \cap \bA_{Q_2}$. It is easy to see that $\bA_P = \bA_P \cap \bA_{Q_1} + \bA_P \cap \bA_{Q_2}$. Thus
\begin{align*}
\bA_P = \cT_{a, b} + \cT_{c, d},
\end{align*}
which is positively existentially definable in $\bK$ with parameters.

\end{proof}

For an element $c \in \bK$, we denote by $\phi(c)$ the statement ``$c$, as an element in $\bK_{\infty} = \GF(\kappa)((t))$, is a square". The following is immediate from Lemma \ref{lem-nonsquares-of-K-infinity}.

\begin{lemma}
\label{lem-squares-in-K-infinity}

Let $c = \dfrac{A}{B} \in \bK$ for some relatively prime polynomials $A, B \in \bA$, and let $a, b$ are the leading coefficients of $A, B$, respectively. Then $\phi(c)$ is equivalent to the following statement:
\begin{itemize}

\item [(i)] the element $\dfrac{a}{b}$ is a square in $\GF(\kappa)$; and

\item [(ii)] $v_{\infty}(c) = \deg(B) - \deg(A)$ is even.

\end{itemize}
\end{lemma}

\begin{remark}
\label{rem-squares-in-K-infinity}

The element $\dfrac{a}{b}$ is called the \textbf{leading coefficient of $c$}. Note that $A, B$ are uniquely determined up to scalar factors in $\GF(\kappa)$. Thus the leading coefficient of $c$ is well-defined.

\end{remark}

Let $\epsilon$ be a nonsquare element in $\GF(\kappa)$, and let $\gamma(a, b)$ denote the following formula
\begin{align}
\label{e-gamma-formula}
&\exists c, d, \delta \big(\text{``$c, d$ are of opposite parity in degree"}   \nonumber \\
& \land [\{\phi(c) \land a = \epsilon c \land b = \delta d \land \psi(\delta) \land \delta \ne 0\}  \nonumber \\
&\lor \{\phi(d) \land b = \epsilon d \land a = \delta c \land \psi(\delta) \land \delta \ne 0\} ]\big),
\end{align}
where $\psi(\cdot)$ denotes the formula (\ref{e-psi(x)-formula-in-koe-2002}).

\begin{remark}

$\psi(\delta)$ in the above formula assures that $\delta$ is an element in $\GF(\kappa)$.

\end{remark}

Set
\begin{align}
\label{e-definition-of-the-important-set-D}
\cD = \{(a, b) \in \bK \times \bK \; | \; \gamma(a, b)\}.
\end{align}

\begin{theorem}
\label{thm-main-thm1}

\begin{align*}
\bA \cup \bA_{\infty} = \bigcap_{(a, b)\in \cD}\widetilde{\cR_{a, b}},
\end{align*}
where $\widetilde{\cR_{a, b}} = \bigcup_{P \in \Delta_{a, b} \bigcap \left(\bP(a) \bigcup \bP(b)\right)}\bA_P$ is defined in \ref{itm: N6}.

\end{theorem}

\begin{proof}

By Proposition \ref{pro-description-of-Delta-a-b}, $\Delta_{a, b} \cap(\bP(a) \cup \bP(b)) = \Delta_{a, b}$, and thus
\begin{align}
\label{e-eqn3-in-main-thm1}
\widetilde{\cR_{a, b}} = \bigcup_{P \in \Delta_{a, b}}\bA_P.
\end{align}

We now show that if $(a, b) \in \cD$, where $\cD$ is defined by (\ref{e-definition-of-the-important-set-D}), then $\Delta_{a, b}$ is nonempty.

By Lemma \ref{lem-nonsquares-of-K-infinity}, any nonsquare element in $\bK_{\infty} = \GF(\kappa)((t))$ admits exactly one of the forms $\dfrac{\alpha^2}{t}$, $h\alpha^2$, or $\dfrac{h\alpha^2}{t}$ for some element $\alpha \in \bK_{\infty}$ and $h \in \GF(\kappa)$. Thus for any pair $(a, b) \in \cD$, viewed as an element of $\bK_{\infty} \times \bK_{\infty}$, there are at most $9$ possible classes for $(a, b)$ modulo squares of $\bK_{\infty}$:
\begin{align*}
&\left(\dfrac{1}{t}, \dfrac{1}{t}\right), \left(\dfrac{1}{t}, h\right), \left(\dfrac{1}{t}, \dfrac{h}{t}\right) \\
&\left(h, \dfrac{1}{t}\right), \left(h, \ell\right), \left(h, \dfrac{\ell}{t}\right) \\
&\left(\dfrac{h}{t}, \dfrac{1}{t}\right), \left(\dfrac{h}{t}, \ell\right), \left(\dfrac{h}{t}, \dfrac{\ell}{t}\right)
\end{align*}
for some nonsquare elements $h, \ell \in \GF(\kappa)$. By the underlying conditions $\gamma(a, b)$ of $\cD$ (see (\ref{e-gamma-formula})), we see that out of the $9$ possible classes for $(a, b)$ modulo squares of $\bK_{\infty}$ above, only four classes are allowed for $(a, b)$ after reducing it modulo squares of $\bK_{\infty}$: $\left(h, \dfrac{\ell}{t}\right)$, $\left(\dfrac{h}{t}, \ell\right)$, $\left(\dfrac{1}{t}, h\right)$, and $\left(h, \dfrac{1}{t}\right)$ for some nonsquare elements $h, \ell$ in $\GF(\kappa)$. By Lemma \ref{lem-ramified-quaternion-algebra-over-K-infinity}, both $\bH_{h, \frac{1}{t}} \otimes \bK_{\infty}$ and $\bH_{h, \frac{\ell}{t}} \otimes \bK_{\infty}$ are nonsplit (or ramified). Since $\bH_{a, b}\otimes \bK_{\infty}$ is isomorphic to $\bH_{\square_1a, \square_2b}\otimes \bK_{\infty}$ for any squares $\square_1, \square_2 \in \bK_{\infty}$. We conclude that $\bH_{a, b}\otimes \bK_{\infty}$ is nonsplit (or ramified), and thus $\infty \in \Delta_{a, b}$ whenever $(a, b) \in \cD$.

By an analogue of the Hilbert Reciprocity Law \ref{thm-analogue-of-the-Hilbert-reciprocity-law}, we deduce that there exists a non-finite prime $Q \in \bP$ such that $\bH_{a, b}$ is nonsplit (or ramified) at $Q$. Thus $\{\infty, Q\} \subseteq \Delta_{a, b}$ for any $(a, b) \in \cD$. Since $\bA$ is a subring of $\bA_Q$, and $\bA_Q \subseteq \widetilde{\cR_{a, b}}$, we see that
\begin{align*}
\bA \cup \bA_{\infty} \subseteq \widetilde{\cR_{a, b}}
\end{align*}
for every $(a, b) \in \cD$, and thus
\begin{align}
\label{e-eqn4-main-thm1}
\bA \bigcup \bA_{\infty} \subseteq \bigcap_{(a, b) \in \cD}\widetilde{\cR_{a, b}}.
\end{align}

We now show that the reverse inclusion of (\ref{e-eqn4-main-thm1}) also holds. Indeed, let $\epsilon$ be a nonsquare element in $\GF(\kappa)^{\times}$ in the formula $\gamma$ given by (\ref{e-gamma-formula}). Take an arbitrary non-infinite prime $P \in \bP$. By Lemma \ref{lem-quaternion-algebra-ramified-exactly-at-a-given-prime-P-and-infinity-main-lem1}, there exists a monic prime $Q \in \bA$ such that $P$ and $Q$ are of opposite parity in degree and the quaternion algebra $\bH_{a_P, b_P}$ is ramified exactly at $P$ and $\infty$, where $a_P = \epsilon P$ and $b_P = \epsilon Q$. Thus $\Delta_{a_P, b_P} = \{P, \infty\}$, and therefore
\begin{align*}
\widetilde{\cR_{a_P, b_P}} = \bA_P \bigcup \bA_{\infty}.
\end{align*}

By the definition of $a_P, b_P$, we see that $\gamma(a_P, b_P)$ is true, and thus $(a_P, b_P) \in \cD$ for every non-infinite prime $P \in \bP$. Thus the pairs $(a_P, b_P) \in \cD$ defined above are in one-to-one correspondence with the set of non-infinite primes $P \in \bP$, and therefore
\begin{align*}
\bigcap_{(a, b) \in \cD}\widetilde{\cR_{a, b}} &\subseteq \bigcap_{(a_P, b_P) \in \cD}\widetilde{\cR_{a_P, b_P}} = \bigcap_{P \in \bP, P \ne \infty}\widetilde{\cR_{a_P, b_P}} \\
&= \bigcap_{P \in \bP, P \ne \infty} \left(\bA_P \bigcup \bA_{\infty}\right) = \left(\bigcap_{P \in \bP, P \ne \infty}\bA_P\right) \bigcup \bA_{\infty}.
\end{align*}

By Lemma \ref{lem-intersection-of-valuation-rings} that
\begin{align*}
\bigcap_{P \in \bP, P \ne \infty}\bA_P = \bA,
\end{align*}
and thus
\begin{align*}
\bigcap_{(a, b) \in \cD}\widetilde{\cR_{a, b}} \subseteq \bA \bigcup \bA_{\infty}.
\end{align*}
Combining the above inclusion with (\ref{e-eqn4-main-thm1}), the theorem follows.

\end{proof}

Adapting the proof of Lemma 2.8 given in Tyrrell \cite{tyrell} and using Koenigsmann's theorem that we recall in Subsection \ref{subsec-koe-2002-defining-transcendentals}, we prove that the set $\cD$ is existentially definable.

\begin{lemma}
\label{lem-cD-is-existentially-definable}

Let $\epsilon$ be a nonsquare element in $\GF(\kappa)$, and let $\gamma(a, b)$ be the formula defined by (\ref{e-gamma-formula}), i.e., the formula of the form
\begin{align*}
&\exists c, d, \delta \big(\text{``$c, d$ are of opposite parity in degree"}   \nonumber \\
& \land [\{\phi(c) \land a = \epsilon c \land b = \delta d \land \psi(\delta) \land \delta \ne 0\}  \nonumber \\
&\lor \{\phi(d) \land b = \epsilon d \land a = \delta c \land \psi(\delta) \land \delta \ne 0\} ]\big),
\end{align*}
where for each $c \in \bK$, $\phi(c)$ denotes ``$c$, as an element of $\bK_{\infty} = \GF(\kappa)((1/t))$, is a square", and for each $\delta \in \bK$, $\psi(\delta)$ is the formula defined by (\ref{e-psi(x)-formula-in-koe-2002}). Then $\gamma(a, b)$ is equivalent to an existential formula, and defined by at most $32$ existential quantifiers. In particular, $\cD$ is an existentially definable set.

\end{lemma}

\begin{proof}

We essentially follow the proof of Lemma 2.8 in Tyrrell \cite{tyrell}. Let $\gamma_1(a, b)$ denote the formula
\begin{align*}
&\exists c, d, \delta \big(\text{``$c, d$ are of opposite parity in degree"}    \\
& \land (\phi(c) \land a = \epsilon c) \land (b = \delta d \land \psi(\delta) \land \delta \ne 0)\big),
\end{align*}
and let $\gamma_2(a, b)$ denote the formula
\begin{align*}
&\exists c, d, \delta \big(\text{``$c, d$ are of opposite parity in degree"}    \\
& \land (\phi(d) \land b = \epsilon d) \land (a = \delta c \land \psi(\delta) \land \delta \ne 0)\big).
\end{align*}
Hence
\begin{align*}
\gamma(a, b) = \gamma_1(a, b) \lor \gamma_2(a, b).
\end{align*}

We first consider $\gamma_1(a, b)$. Following the proof of Lemma 2.8 in Tyrrell \cite{tyrell}, the subformula $\lambda_1(a, b)$ of $\gamma_1(a, b)$, say $\exists c, d\big(\text{``$c, d$ are of opposite parity in degree"} \land (\phi(c) \land a = \epsilon c)\big)$ is equivalent to the formula of the form $\exists u, v(\deg(A) \ge \deg(B))$, where $A, B$ are certain terms. The formula ``$\deg(A) \ge \deg(B)$" is equivalent to ``$v_{\infty}(B/A) \ge 0$". By Theorem \ref{thm-integrality-at-a-prime-for-K}, the set $\{\alpha \in \bK \; | \; v_{\infty}(\alpha) \ge 0\}$ is existentially definable, and defined by $9$ quantifiers. Thus $\lambda_1(a, b)$ is existentially definable, and requires $11$ quantifiers to define.

Since $\psi(\delta)$ is existentially definable (see (\ref{e-psi(x)-formula-in-koe-2002})), and requires $4$ quantifiers to define, the subformula $\lambda_2(a, b)$ of $\gamma_1(a, b)$ of the form $\exists d, \delta(b = \delta d \land \psi(\delta) \land \delta \ne 0)$ is obviously existentially definable, and defined by $6$ quantifiers. Since $\lambda_1(a, b)$ and $\lambda_2(a, b)$ have a common quantifier $\exists d$, the formula $\gamma_1(a, b) = \lambda_1(a, b) \lor \lambda_2(a, b)$ is existentially definable and requires $16$ quantifiers.

Using the same arguments, we deduce that $\gamma_2(a, b)$ is also existentially definable, and defined by $16$ parameters. Thus $\gamma(a, b) = \gamma_1(a, b) \lor \gamma_2(a, b)$ is existentially definable and requires at most $32$ quantifiers to define.

\end{proof}

\begin{corollary}
\label{cor-main-corollary1}

There is a universal-existential definition of $\bA = \GF(\kappa)[t]$ in $\bK = \GF(\kappa)(t)$ by a formula with $110$ universal quantifiers followed by $4$ existential quantifiers.

\end{corollary}

\begin{proof}

By Theorem \ref{thm-main-thm1}, we see that
\begin{align}
\label{e-eqn1-main-corollary1}
\alpha(t) \in \bA \cup \bA_{\infty} \Leftrightarrow \forall a, b\left((a, b) \not\in \cD \lor \alpha(t) \in \widetilde{\cR_{a, b}}\right).
\end{align}
By Corollary \ref{cor-R-tilde-a-b-is-Diophantine}, $\widetilde{\cR_{a, b}}$ is universally definable in $\bK$ and defined by at most $67$ quantifiers. Since $\cD$ is existentially definable by a formula with at most $32$ quantifiers, (\ref{e-eqn1-main-corollary1}) is a universal formula for $\bA \cup \bA_{\infty}$ with at most $2 + 67 + 32 = 101$ quantifiers. We denote this universal formula by $\Sigma(\alpha(t))$. Thus we can define $\bA$ in $\bK$ as follows:
\begin{align}
\label{e-eqn2-main-corollary1}
\alpha(t) \in \bA = \GF(\kappa)[t] \Leftrightarrow \Sigma(\alpha(t)) \land \left(\deg(\alpha(t)) > 0 \lor \deg(\alpha(t) = 0 \right).
\end{align}

By Theorem \ref{thm-integrality-at-a-prime-for-K}, $\{\alpha(t) \in \bK\; |\; \deg(\alpha(t)) > 0\} =\{\alpha(t)\in \bK \; | \; v_{\infty}(\alpha(t)) < 0\}$ is universally definable by a formula with $9$ quantifiers,

Note that ``$\deg(\alpha(t)) = 0$" is euivalent to ``$\alpha(t) \in \bK$ belongs in $\GF(\kappa)$" that is existentially definable in $\bK$ by the formula $\psi(\alpha(t))$ with $4$ existential quantifiers (see Corollary \ref{cor-Koe-2002}), where $\psi(\alpha(t))$ is the formula defined by (\ref{e-psi(x)-formula-in-koe-2002}). Therefore we deduce from (\ref{e-eqn2-main-corollary1}) that $\bA = \GF(\kappa)[t]$ is definable in $\bK = \GF(\kappa)(t)$ by a formula with at most $101 + 9 = 110$ universal quantifiers followed by $4$ existential quantifiers.

\end{proof}

We recall the following result that is due to R. Robinson \cite{RROB}.

\begin{theorem}
\label{thm-r.robinson-1951-theorem}

If $\bF$ is a field of characteristic $0$, the rational number field $\bQ$ is elementarily definable in the polynomial ring $\bF[t]$ by a first-order formula independent of $\bF$.

\end{theorem}

\begin{proof}

See Robinson \cite{RROB} or Jensen--Lenzing \cite[Corollary 3.7, p.36]{JL89}.

\end{proof}

Using Robinson's theorem and Corollary \ref{cor-main-corollary1}, we prove the undecidability of the full first-order theory of $\GF(\kappa)(t)$ when the finite fields $\bF_{q_s}$ are of distinct characteristics.

\begin{corollary}
\label{cor-undecidability-of-GF(kappa)(t)}

Suppose that the finite fields $\bF_{q_s}$ are of distinct characteristics, so the ultra-finite field $\GF(\kappa)$ is of characteristic $0$. Then the full first-order theory of $\GF(\kappa)(t)$ is undecidable.

\end{corollary}

\begin{proof}

By J. Robinson's theorem \cite{rob49}, the full first-order theory of the rational number field $\bQ$ is undecidable. Thus we deduce from Theorem \ref{thm-r.robinson-1951-theorem} that the full first-order theory of the polynomial ring $\GF(\kappa)[t]$ is undecidable, and it thus follows from Corollary \ref{cor-main-corollary1} that the full first-order theory of $\GF(\kappa)(t)$ is undecidable.

\end{proof}

\section*{Acknowledgements}

I would like to thank Anand Pillay for useful correspondence about ultraproducts and model theory, and am grateful for his pointing out to me the reference \cite{JL89}.

\end{document}